\documentclass[11pt,leqno]{amsart}
\usepackage{amssymb} 
\usepackage{mathrsfs}
\usepackage[all]{xy}
\usepackage{color}
\usepackage{soul}
\usepackage{hyperref}

\newcommand\N{{\mathbb N}}
\newcommand\R{{\mathbb R}}
\newcommand\T{{\mathbb T}}
\newcommand\C{{\mathbb C}}

\newcommand\Sp{{\mathbb S}}
\newcommand\E{{\mathbb E}}

\def\AA{{\mathcal A}}
\def\BB{{\mathcal B}}
\def\CC{{\mathcal C}}

\def\EE{{\mathcal E}}

\def\II{{\mathcal I}}

\def\LL{{\mathcal L}}

\def\OO{{\mathcal O}}
\def\PP{{\mathcal P}}

\def\RR{{\mathcal R}}

\def\UU{{\mathcal U}}

\def\ZZ{{\mathcal Z}}

\def\BBB{{\mathscr B}}
\def\CCC{{\mathscr C}}

\def\eps{{\varepsilon}}

\def\Nt{|\hskip-0.04cm|\hskip-0.04cm|}

\newtheorem{theo}{Theorem}
\newtheorem{prop}[theo]{Proposition}
\newtheorem{lem}[theo]{Lemma}
\newtheorem{cor}[theo]{Corollary}
\newtheorem{rem}[theo]{Remark}

\newtheorem{ass}[theo]{Assumptions}
\newtheorem{ass1}[theo]{Assumption}

\newcommand{\beqn}{\begin{equation}}
\newcommand{\eeqn}{\end{equation}}
\newcommand{\bear}{\begin{eqnarray}}
\newcommand{\eear}{\end{eqnarray}}
\newcommand{\bean}{\begin{eqnarray*}}
\newcommand{\eean}{\end{eqnarray*}}

\title[Boltzmann equation for granular media]{Boltzmann equation for granular media with thermal force in a weakly inhomogeneous setting}

\begin{document}

\author{{\sc Isabelle Tristani}}
\address{CEREMADE, Universit\'e Paris IX-Dauphine,
Place du Mar\'echal de Lattre de Tassigny, 75775 Paris Cedex 16, France. 
E-mail: {\tt tristani@ceremade.dauphine.fr }}

\date\today

\maketitle

\begin{abstract} 
In this paper, we consider the spatially inhomogeneous diffusively driven inelastic Boltzmann equation in different cases:  the restitution coefficient can be constant or can depend on the impact velocity (which is a more physically relevant case), including in particular the case of viscoelastic hard spheres. In the weak thermalization regime, i.e. when the diffusion parameter is sufficiently small, we prove existence of global solutions considering the close-to-equilibrium regime as well as the weakly inhomogeneous regime in the case of a constant restitution coefficient. It is the very first existence theorem of global solution in an inelastic ``collision regime'' (that is excluding \cite{AR} where an existence theorem is proven in a near to the vacuum regime). We also study the long-time behavior of these solutions and prove a convergence to equilibrium with an exponential rate. The basis of the proof is the study of the linearized equation. We obtain a new result  on it, we prove existence of a spectral gap in weighted (stretched exponential and polynomial) Sobolev spaces and a result of exponential stability for the semigroup generated by the linearized operator. To do that, we develop a perturbative argument around the spatially inhomogeneous equation for elastic hard spheres and we take advantage of the recent paper \cite{GMM}
where this equation has been considered. We then link the linearized theory with the nonlinear one in order to handle the full non-linear problem thanks to new bilinear estimates on the collision operator that we establish. As far as the case of a constant coefficient is concerned, the present paper largely improves similar results obtained in \cite{MM2} in a spatially homogeneous framework. Concerning the case of a non-constant coefficient, this kind of results is new and we use results on steady states of the linearized equation from \cite{AL3}.
\end{abstract}

\vspace{1cm}
\textbf{Mathematics Subject Classification (2010)}: 76P05 Rarefied gas flows, Boltzmann equation [See also 82B40, 82C40, 82D05]; 76T25 Granular flows [See also 74C99, 74E20]; 47H20 Semigroups of nonlinear operators [See also 37L05, 47J35, 54H15, 58D07].

\vspace{0.3cm}
\textbf{Keywords}: Inelastic Boltzmann equation; granular media; viscoelastic hard spheres; small diffusion parameter; elastic limit; perturbation; spectral gap; exponential rate of convergence; long-time asymptotic.

\vspace{0.5cm}
\tableofcontents


\section{Introduction} 
\label{sec:intro}
\setcounter{equation}{0}
\setcounter{theo}{0}


\subsection{The model} \label{subsec:model}
We investigate in the present paper the Cauchy theory associated to the spatially inhomogeneous diffusively driven inelastic Boltzmann equation for hard spheres interactions and constant or non-constant restitution coefficient. More precisely, we consider hard spheres particles described by their distribution density $f = f(t,x,v) \geq 0$ undergoing inelastic collisions in the torus in dimension $d=3$. The spatial coordinates are $x \in \T^3$ (3-dimensional flat torus) and the velocities are $v \in \R^3$. In the model at stake, the inelasticity is characterized by the so-called normal restitution coefficient $e_\lambda(\cdot):=  e(\lambda�\, \cdot)$ which can be, as opposed to most of previous contributions on the subject, constant or non-constant. The distribution $f$ satisfies the following equation:
\beqn \label{eq:Bol1}
\partial_t f =  Q_{e_\lambda} (f,f) + \lambda^\beta \, \Delta_v f - v \cdot  \nabla_x f \quad \text{with} \quad
\beta = \left\{
\begin{aligned}
&1 \quad \text{if} \quad e(\cdot) = 1 - \lambda, \\
&\gamma>0 \quad \text{if} \quad e(r) \underset{0}{\sim} 1-a \, r^\gamma, \, \, a>0.
\end{aligned}
\right.
\eeqn

The term $\lambda^\beta \, \Delta_v f $ represents a constant heat bath which models particles uncorrelated random accelerations between collisions. The quadratic collision operator $Q_{e_\lambda}$ models the interactions of hard spheres by inelastic binary collisions. It is important to emphasize the fact that contrary to the case of elastic collisions (which correspond to $\lambda =0$ in \eqref{eq:Bol1}) for which both momentum and energy are preserved during collisions, in the inelastic case, momentum is preserved but there is a loss of energy during collisions. The restitution coefficient quantifies the loss of relative normal velocity of a pair of colliding particles after the collision with respect to the impact velocity. More precisely, if $v$~and~$v_*$ (resp. $v'$ and $v'_*$) denote the velocities of a pair of particles before (resp. after) collision, we have the following equalities
\beqn \label{eq:coll}
\left\{
\begin{aligned}
&v+v_* = v'+ v'_*, \\
&(u' \cdot \widehat{n}) = -(u \cdot \widehat{n}) \, e_\lambda (u \cdot \widehat{n}),
\end{aligned}
\right.
\eeqn
where 
$$
u=v-v_*, \quad u' = v'- v'_*,
$$
denote respectively the relative velocity before and after collision, $e:=e(|u \cdot \widehat{n}|)$ is such that $0 \leq e \leq 1$ and the direction of the unitary vector $\widehat{n} \in \Sp^2$ is accordingly to the impact direction, meaning that $\widehat{n}$ stands for the unit vector that points from the $v$-particle center to the $v_*$-particle center at the instant of impact. Assuming the granular particles to be perfectly smooth hard spheres of mass $m = 1$, the velocities after collision $v'$ and $v'_*$ are then given, in virtue of~(\ref{eq:coll}), by
\beqn \label{eq:velo}
v' = v - \frac{1+e_\lambda}{2} \, (u \cdot \widehat{n}) \,  \widehat{n}, \quad v'_* = v_* + \frac{1+e_\lambda}{2} \, (u \cdot \widehat{n})\,  \widehat{n}.
\eeqn

There exist others possible parametrizations of post-collisional velocities. We here introduce another parametrization that shall be more convenient in the sequel. If $v$ and $v_*$ are the velocities of two particles with $v \neq v_*$, we set $\widehat{u} = u / |u|$. We then perform in~(\ref{eq:velo}) the change of unknown $ \sigma = \widehat{u} - 2 \, (\widehat{u} \cdot \widehat{n}) \, \widehat{n} \in \Sp^2$, it gives us an alternative parametrization of the unit sphere $\Sp^2$. The impact velocity then writes $|u \cdot \widehat{n} | = |u| \sqrt{ \frac{ 1 - \widehat{u} \cdot \sigma}{2}}$ and the post-collisional velocities $v'$ and $v'_*$ are then given by
\beqn \label{eq:velo2}
v' = v - \frac{ 1+ e_\lambda}{2} \,  \frac{u - |u|�\sigma}{2}, \quad v'_* = v_* + \frac{ 1+ e_\lambda}{2} \,  \frac{u - |u|�\sigma}{2}.
\eeqn
The loss of energy during collisions then comes down to the following inequality:
\beqn \label{eq:lossenergy} 
|v'|^2 + |v'_*|^2 - |v|^2 - |v_*|^2 = - |u|^2 \, \frac{1-\widehat{u} \cdot \sigma}{4} \, \left( 1 - e_\lambda\left( |u| \, \sqrt{(1- \widehat{u} \cdot \sigma)/2} \right)^2 \right)<0.
\eeqn

The representation of post-collisional velocities through \eqref{eq:velo2} allows us to give a definition of the Boltzmann collision operator in weak form by
\beqn \label{eq:collop}
\begin{aligned}
& \int_{\R^3} Q_{e_\lambda} (g,f) \, \psi \, dv = \int_{\R^3 \times \R^3 \times \Sp^2} g(v_*) f(v) \left[\psi(v') \, - \, \psi(v) \right] |v-v_*| \, d\sigma \, dv_* \, dv \\
&\qquad ={1 \over 2} \int_{\R^3 \times \R^3 \times \Sp^2} g(v_*) f(v) \left[\psi(v'_*) + \psi(v')  -  \psi(v_*) - \psi(v) \right] |v-v_*| \, d\sigma \, dv_* \, dv,
\end{aligned}
\eeqn
for any $\psi=\psi(v)$ a suitable regular test function. 
As can be easily seen in the second weak formulation in \eqref{eq:collop}, the operator $Q_{e_\lambda}$ preserves mass and momentum, and since the Laplacian also does so, the equation preserves mass and momentum. However, energy is not preserved either by the collisional operator (which tends to cool down the gas because of~(\ref{eq:lossenergy})) or by the diffusive operator (which warms it up).

\medskip

\subsection{Function spaces}
For some given Borel weight function $m>0$ on $\R^3$, let us define $L^q_v L^p_x(m)$, $1 \le p,q \le +\infty$, as the Lebesgue space associated to the norm 
$$
\| h \|_{L^q_v L^p_x(m)} = \| \| h(\cdot, v)\|_{L^p_x} \, m(v) \|_{L^q_v}.
$$
We also consider the standard higher-order Sobolev generalizations $W^{\sigma, q}_v W^{s,p}_x(m)$ for \linebreak $\sigma, s \in \N$ defined by the norm 
$$
\| h \|_{W^{\sigma, q}_v W^{s,p}_x(m)} = \sum_{0\leq s' \leq s, \, 0 \leq \sigma' \leq \sigma, \, s'+ \sigma' \leq \max(s,\sigma)} 
\| \| \nabla_x^{s'} \nabla_v^{\sigma'} h(\cdot, v)\|_{L^p_x} \, m(v) \|_{L^q_v}.
$$
This definition reduces to the usual weighted Sobolev space $W^{s,p}_{x,v}(m)$ when $q=p$ and $\sigma=s$, and we recall the shorthand notation $H^s = W^{s,2}$.

\medskip
\subsection{Main and known results}
Before stating our main result, let us recall some facts on equilibriums of our equation. We know that there exists $G_\lambda=G_\lambda(v)$ a space homogeneous solution of the stationary equation
$$
Q_{e_\lambda}(f,f) +\lambda^\gamma \Delta_v f = 0 
$$
with mass $1$ and vanishing momentum. Moreover, $G_\lambda$ is unique for $\lambda$ close enough to $0$ and satisfies the following estimates at infinity for some $A$, $M>0$:
\beqn \label{eq:Glambda}
\int_{\R^3} G_\lambda (v) \,  e^{A \, |v|^{3/2}} \, dv \leq M, \quad \int_{\R^3} G_\lambda (v) \,  e^{|v|^2/2} \, dv = \infty.
\eeqn
We precise here that a large literature has been devoted to the study (existence, uniqueness and growth properties) of self-similar profiles (resp. stationary solutions) for freely cooling (resp. driven by a thermal bath) inelastic hard spheres with a constant restitution coefficient (see the papers of Bobylev et al. \cite{BGP}, Gambda et al. \cite{GPV}, Mischler and Mouhot \cite{MM,MM2,MM3}) and with a non-constant one (see the paper of Alonso and Lods~\cite{AL3}). 

\subsubsection{Cauchy theory}
Our main result is the proof of existence of solutions for the non-linear problem~(\ref{eq:Bol1}) as well as stability and relaxation to equilibrium for these solutions. In both constant and non-constant cases, we are able to prove an existence theorem in a close-to-equilibrium regime; the main result in this regime reads as follows (see Theorem~\ref{theo:existence} for a precise statement):

\begin{theo}
Consider $\EE = W^{s,1}_x L^1_v \left(e^{b\langle v \rangle^\beta}\right)$ and $\EE^1 = W^{s,1}_x L^1_v \left(\langle v \rangle e^{b\langle v \rangle^\beta}\right)$ where $b>0$, $\beta \in (0,1)$ and $s>6$. For a convenient non-constant restitution coefficient $e$ or for a constant one, for $\lambda$ small enough, and for an initial datum $f_{in} \in \EE$ close enough to the equilibrium $G_\lambda$, there exists a unique global solution $f \in L^\infty_t(\EE) \cap L^1_t(\EE^1)$ to (\ref{eq:Bol1}) which furthermore satisfies
$$
\forall \, t \geq 0, \quad \|f_t - G_\lambda\|_{\EE} \leq C \, e^{-\widetilde{\alpha} \, t} \, \|f_{in} - G_\lambda\|_{\EE}
$$
for some constructive constants $C$ and $\widetilde{\alpha} >0$. 
\end{theo}
\noindent Moreover, in the case of a constant restitution coefficient,
we are able to get a similar result in a weakly inhomogeneous regime. Let us explain why we do not obtain such a result for a non-constant restitution coefficient. It is due to the fact that no result on the long-time behaviour of solutions to the spatially homogeneous problem is available for general initial data (far from equilibrium) contrary to the case of constant restitution coefficient. The main result concerning the weakly inhomogeneous regime reads as follows (see Theorem~\ref{theo:weak2} for a precise statement):

\begin{theo}
Consider $\EE = W^{s,1}_x L^1_v \left(e^{b\langle v \rangle^\beta}\right)$ and $\EE^1 = W^{s,1}_x L^1_v \left(\langle v \rangle e^{b\langle v \rangle^\beta}\right)$ where $b>0$, $\beta \in (0,1)$ and $s>6$. For a constant restitution coefficient $e$, for $\lambda$ small enough, and for an initial datum $f_{in} \in \EE$ close enough to a suitably regular spatially homogeneous distribution $g_{in}=g_{in}(v)$, there exists a unique global solution $f \in L^\infty_t(\EE) \cap L^1_t(\EE^1)$ to (\ref{eq:Bol1}) which furthermore satisfies
$$
\forall \, t \geq 0, \quad \|f_t - G_\lambda\|_{\EE} \leq C \, e^{-\widetilde{\alpha} \, t} \, \|f_{in} - G_\lambda\|_{\EE}
$$
for some constructive constants $C$ and $\widetilde{\alpha} >0$. 
\end{theo}

The innovative aspect of this work lies both in the obtained result and in the method of proof. 

On the one hand, concerning the result itself, we have to underline that it is the first time that an existence result is obtained in the spatially inhomogeneous case in an inelastic ``collision regime'',  moreover, in both cases of constant and non-constant coefficient of inelasticity. 
Let us mention that most of the previous results have been established in an homogeneous framework (see the paper of Gambda et al. \cite{GPV} and the one of Mischler et al. \cite{MMRR} for homogeneous Cauchy theories). For the inhomogeneous inelastic Boltzmann equation, the literature is more scarce; in this respect we mention the work of Alonso~\cite{AR} that treats the Cauchy problem in the case of near-vacuum data and thus in a case where collisions do not play a significant role. It is valuable mentioning that the scarcity of results regarding existence of solutions for the inhomogeneous inelastic case compared to the inhomogeneous elastic case is explained by two facts. In the inhomogeneous elastic case, the Cauchy problem has been handled through two different frameworks that we briefly explain in what follows. First, the theory of perturbative solutions which is based on the spectral study of the linearized associated operator that goes back to Hilbert \cite{H1,H2} Carleman \cite{Carl}, Grad \cite{Grad1,Grad2} and Ukai \cite{Uk} who built the first global solutions of Boltzmann equation. Up to now, it was not possible to look forward to such a strategy in the inelastic case because of the absence of precise spectral study of the linearized problem. The second well-known theory in the elastic case is the one of DiPerna-Lions renormalized solutions \cite{DPL} which is no longer available in the inelastic case due to the lack of entropy estimates for the inelastic Boltzmann equation. 

On the other hand, regarding the method of proof, we first develop a Cauchy theory of perturbative solutions before going into the weakly inhomogeneous framework. As mentioned previously, the core of this theory is the study of the linearized associated problem. Consequently, we make a precise analysis of the linearized equation from a viewpoint of spectral study and with a semigroup approach as explained in the following paragraph. 

\subsubsection{The linearized equation}
We linearize our equation around the equilibrium $G_\lambda$ with the perturbation $f=G_\lambda+h$.
We obtain at first order the linearized equation around the equilibrium $G_\lambda$:
\beqn \label{eq:Bol2}
\partial_t h = \LL_\lambda h := Q_{e_\lambda}(h,G_\lambda) + Q_{e_\lambda}(G_\lambda,h) + \lambda^\gamma \Delta_v h - v \cdot \nabla_x h.
\eeqn
 The lack of spectral results on the linearized inelastic operator $\LL_\lambda$ is explained by the fact that equilibriums of the equation $G_\lambda$ are not anymore explicit and do not decrease as Maxwellians as in the elastic case (see~\eqref{eq:Glambda}). In this elastic case, the spectral study of the linearized associated operator strongly relied, until recently, on symmetry arguments and on the Hilbert structure of $L^2(\mu^{-1/2})$ where $\mu := G_0 = (2 \pi)^{-3/2} e^{-|v|^2/2}$ is the elastic equilibrium of mass $1$, vanishing momentum and energy $3$. In our inelastic case, we lose this convenient structure. Our results are established in a small inelasticity regime (close to the elastic one), our strategy is thus to use a perturbative argument around the elastic case. What is important to highlight here is that, contrary to the elastic case where the equilibriums of the equation are Maxwellians, equilibriums in the inelastic case do not decrease enough to belong to weighted Hilbert spaces of type $L^2(e^{|v|^2/2})=L^2(\mu^{-1/2})$ (see \eqref{eq:Glambda}), the natural space in which the study of the linearized elastic equation is done. It is thus not possible to develop a perturbative theory around the elastic case in this kind of spaces. We thus have to use the recent work \cite{GMM} in which the study of the linearized elastic Boltzmann operator has been developed in larger Banach spaces, it gives explicit spectral gap estimates on the semigroup associated to the linearized non homogeneous operator $\LL_0$ in various Sobolev spaces $W^{s,p}_x W^{\sigma, q}_v (m)$ with polynomial or stretched exponential weights~$m$.
In this kind of spaces which can contain the inelastic equilibriums $G_\lambda$, one could consider to develop a perturbative argument around the elastic case. 
Here is a rough version of the main result that we obtain on the linearized operator (see Theorem \ref{theo:linearmainresult} for a complete version):
\begin{theo}
For $\lambda$ close enough to $0$, the spectrum $\Sigma (\LL_\lambda)$ of $\LL_\lambda$ satisfies the following separation property in $\EE = W^{s,1}_x L^1_v \left( e^{b\langle v \rangle^\beta}\right)$ ($b>0$, $\beta \in (0,1)$ and $s \in \N$):
$$
\Sigma(\LL_\lambda) \cap \left\{ z \in \C, \, \Re e \, z > -\alpha_1 \right\}= \{\mu_\lambda,0 \},
$$
where $\alpha_1$ is the elastic spectral gap. Moreover, $0$ is a four-dimensional eigenvalue (due to the conservation of mass and momentum) and $\mu_\lambda \in \R$, the ``energy" eigenvalue, is a one-dimensional eigenvalue that satisfies $\mu_\lambda<0$ and $\mu_\lambda \xrightarrow[\lambda \to 0]{} 0$.

We have the following estimate on $S_{\LL_\lambda}(t)$, the semigroup generated by $\LL_\lambda$:
\beqn \label{eq:SGdecay}
\forall \, t \geq 0, \quad \|S_{\LL_\lambda}(t)(I -  \Pi_{\LL_\lambda,0} -\Pi_{\LL_\lambda,\mu_\lambda}) \|_{\BBB(\EE)} \leq C e^{-\widetilde{\alpha} t}
\eeqn
for some constructive constants $\widetilde{\alpha}>0$ and $C \ge 1$, with $\Pi_{\LL_\lambda,0}$ (resp. $\Pi_{\LL_\lambda,\mu_\lambda}$) the projection onto the null space of $\LL_\lambda$ (resp. the eigenspace associated to $\mu_\lambda$). 
\end{theo}
This theorem is proven thanks to a perturbative argument around the elastic case in the same line as the one developed by Mischler and Mouhot in \cite{MM,MM2} but we largely improve it in three aspects: we are able to deal with the spatial dependency in the torus; we are able to deal with non-constant restitution coefficients and we are able to obtain a decay estimate on the semigroup \eqref{eq:SGdecay} using the localization of the spectrum.
The third point is crucial in our strategy and we will explain why it is so important in the following paragraph.

\medskip

\subsection{Method of proof}
We now go into more details regarding the strategy of the proof. As explained previously, to develop a Cauchy theory for the equation \eqref{eq:Bol1}, we first study the linearized problem around the equilibrium and establish its asymptotic stability by a perturbation argument which uses the spectral analysis of the linearized elastic Boltzmann equation. 

This perturbative argument is based on several facts that we precise here. The first one comes from \cite{GMM}: the spectrum of the linearized elastic equation is well localized, meaning that it admits a spectral gap in a large class of Sobolev spaces. The second one is that for $\lambda$ small enough,
$$
\LL_\lambda - \LL_0 = \OO(\lambda)
$$
for a suitable operator norm, this kind of estimates relies on accurate estimates on the difference $Q_{e_\lambda} - Q_1$ that we establish. The third one is that the semigroup $S_{\LL_\lambda}$ generated by $\LL_\lambda$ splits as 
$$
S_{\LL_\lambda} = S^1_\lambda + S^2_\lambda, \quad S^1_{\lambda} \simeq e^{t T_\lambda}, \quad T_\lambda \, \,\text{finite dimensional}, \quad S^2_\lambda = \OO(e^{at}), \, \, a<0
$$
still in suitable Sobolev spaces that contain the equilibriums $G_\lambda$. This decomposition of the semigroup is obtained by introducing a suitable splitting of the operator $\LL_\lambda$
$$
\LL_\lambda = \AA_\lambda + \BB_\lambda
$$
where $\BB_\lambda$ enjoys some dissipativity properties and $\AA_\lambda$ some regularity properties. These operators are defined through an appropriate mollification-truncation process, described later on. Such a splitting has to be exhibited in general Sobolev spaces and is one of the main technical issues of this work. Those two facts combined with the well-localization of the spectrum of the elastic linearized operator $\LL_0$ in those Sobolev spaces allow us to deduce some properties on the spectrum of $\LL_\lambda$. In short, the inelastic model is seen as a perturbation of the elastic one and our analyze takes advantage of such a property in order to capture the asymptotic behavior of the related spectral objects (spectrum, spectral projector...). This perturbative argument is also technically challenging and is of broader concern than just the inelastic Boltzmann equation, it can be useful for others type of equations (Fokker-Planck, Boltzmann and Landau...).

From the spectral properties that we are able to get thanks to this perturbation argument, we are then capable to obtain an estimate on the semigroup thanks to a (principal) spectral mapping theorem. This part is fundamental to go back to the non linear problem. Contrary to the homogeneous case in \cite{MM2} where the existence of solutions was already known from \cite{GPV}, in our case, we need to build a solution. To do that, we use an iterative scheme whose convergence is ensured thanks to a priori estimates coming from, among others, estimates on the semigroup of the linearized operator. Another key element for getting those a priori estimates is the proof of new estimates on the bilinear collision operator that we establish in the general inhomogeneous setting. The idea is to prove that for a sufficiently close to the equilibrium initial datum (in the ``linearization trap''), the nonlinear part of the equation is negligible with respect to the linear part which, consequently, dictates the dynamic. We can thus use the result that we obtained on the linearized semigroup and recover an exponential decay to equilibrium for the nonlinear problem.

\medskip

\subsection{Physical and mathematical motivations} \label{subsec:phymath}
We won't enter into details concerning the physical introduction to granular gases, we refer to the works of Brillantov and P{\"o}schel~\cite{BP} and Cercignani~\cite{Cer}. To put it briefly, granular flows have become of major interest in physical research for several decades. This analysis is based on kinetic theory for some regimes of dilute and rapid flows. However, the mathematical study started later, in the end of the 1990 decade. Once again, we do not give an exhaustive list of references for the mathematical introduction to this theory, we only refer  to the papers of Mischler et al.~\cite{MMRR,MM3}. As explained in the latter, granular gases are made up of grains of macroscopic size whith contact collisional interactions, assuming that there are no self-interaction mechanisms such as gravitation, electromagnetism... As a consequence, it is natural to suppose that the binary interaction between grains is that of inelastic hard spheres with no loss of ``tangential relative velocity'' (according to the impact direction) and a loss in ``normal relative velocity''. This loss is quantified in some (normal) restitution coefficient $e$ introduced in Subsection \ref{subsec:model} which is either assumed to be constant as a first approximation or can be more intricate: for instance it is a function of the modulus $|v'-v|$ of the normal relative velocity in the case of ``visco-elastic hard spheres'' (see~\cite{AL},~\cite{AL2},~\cite{AL3} and~\cite{BP}). In this paper, we consider those two cases. More specifically, in the non-constant case, the main assumption on $e(\cdot)$ we shall need is listed in the following.

\begin{ass} \label{ass} 
\
\begin{enumerate} 
\item \label{ass:1} The mapping $r \rightarrow e(r)$ from $\R^+$ to $(0,1]$ is absolutely continuous and non-increasing.
\item \label{ass:2} The mapping $r \rightarrow r \, e(r)$ is strictly increasing on $\R^+$.
\item \label{ass:3} There exist $a,b>0$ and $\overline{\gamma} > \gamma > 0$ such that 
$$
\forall \, r \geq 0, \quad |e(r) - 1 + a \, r^\gamma | \leq b \, r^{\overline{\gamma}} .
$$ 
\end{enumerate}
\end{ass}
\noindent Let us now describe the most physically relevant model we have in mind and that fulfills Assumptions \ref{ass}, the one corresponding to viscoelastic hard spheres for which the restitution coefficient has been derived by Schwager and P{\"o}schel in~\cite{SP}. For this model, $e(\cdot)$ can be represented using an infinite expansion series as follows:
\beqn \label{eq:visco}
e ( |u \cdot \widehat{n}|) = 1 + \sum_{k=1}^\infty (-1)^k \, a_k |u \cdot \widehat{n} |^{k/5}, \quad u \in \R^3, \quad \widehat{n} \in \Sp^2
\eeqn
where $a_k > 0$ for any $k \in \N$ are parameters depending on the material viscosity. One can show that in this case, $e(\cdot)$ satisfies Assumptions~\ref{ass} with $\gamma = 1/5$ and $\overline{\gamma}=2/5$ for~\eqref{ass:3}. In the case of a non-constant restitution coefficient, this is the principal example of applications of the results in the paper, though, the range of application of our results is more general. 

\smallskip
We now explain our motivations to restrict our study to the case of a small diffusion parameter (weak thermalization regime), corresponding to small inelasticity which come both from mathematics and physics. 
First, the assumption of small inelasticity is adequate from the viewpoint of the regime of validity of kinetic theory. Indeed, the validity of Boltzmann's theory heavily relies on the molecular chaos assumption, and as explained in~\cite{BP} for instance, the more inelasticity, the more correlations between grains are created during the binary collisions. Secondly, the case of small inelasticity is interesting since it allows to use expansions around the elastic case and it also gives rise to another issue of interest, the link between the inelastic case (dissipative at the microscopic level) and the elastic one (``Hamiltonian'' at the microscopic level). This also explains (cf. \cite{BP}) why this case has been considered so largely. Finally, this case of a small inelasticity is justifiable regarding applications, since it applies to interstellar dust clouds in astrophysics, or sands and dusts in earth-bound experiments, and more generally to visco-elastic hard spheres whose restitution coefficient is not constant but close to $1$ on the average, as explained previously.

%
  \smallskip
Finally, we clarify why studying the rescaled equation \eqref{eq:Bol1} is relevant in the case of weak thermalization regime. The associated stationary equation before rescaling is given by 
\beqn \label{eq:nonrescaled}
Q_e(f,f) + \mu \, \Delta_v f - v \cdot \nabla_x f =0
\eeqn
for some positive thermalization coefficient $\mu >0$. We then introduce the rescaled distribution $g_\lambda(x,v) := \lambda^3 \, f(x,\lambda v)$ if $f$ is a solution of (\ref{eq:nonrescaled}) of mass $\rho$. Using the following equalities which hold for any $x \in \T^3$ and $v \in \R^3$,
$$
\begin{aligned}
\lambda^2 \, Q_e(f,f) (x,\lambda v) &= Q_{e_\lambda}(g_\lambda, g_\lambda) (x,v), \\ 
\lambda^5 \, (\Delta_v f)(x,\lambda v) &= \Delta_v g_\lambda (x,v), \\
\lambda^3 \, v \cdot \nabla_x f (x, \lambda v) &= v \cdot \nabla_x g_\lambda(x,v),
\end{aligned}
$$
we obtain that $g_\lambda$ satisfies
\beqn \label{eq:rescaled}
Q_{e_\lambda} (g_\lambda, g_\lambda) + \frac{\mu}{\lambda^3} \, \Delta_v g_\lambda- v \cdot \nabla_x g_\lambda = 0.
\eeqn
Let us notice that this scaling preserves mass and momentum and moreover, $e_\lambda(r)$ tends to $1$ as $\lambda$ goes to $0$, the elastic restitution coefficient. We expect that formally, as $\lambda$ goes to $0$, 
$$
Q_{e_\lambda}(f,f) \simeq Q_1(f,f)
$$
and thus that, as $\lambda$ goes to $0$, the dissipation of energy vanishes. 
We see that if $\mu>0$ is fixed, then the second term of (\ref{eq:rescaled}) becomes infinite in the limit $\lambda \rightarrow 0$. We thus have to choose $\mu:=\mu_\lambda$ such that $\mu_\lambda \, \lambda^{-3}$ tends to $0$ as $\lambda$ goes to $0$. Such as in \cite{AL3}, we can compute a parameter $\mu_\lambda$ such that the energy 
$$
\EE_\lambda := \frac{1}{\rho} \, \int_{\T^3 \times \R^3} g_\lambda (x,v) \, |v|^2 \, dx \, dv
$$
is kept of order one in the limit $\lambda \rightarrow 0$, which gives $\mu = \mu_\lambda = \lambda^{3+\gamma}$. Equation (\ref{eq:rescaled}) hence becomes
$$
Q_{e_\lambda} (g_\lambda, g_\lambda) +\lambda^{\gamma} \, \Delta_v g_\lambda- v \cdot \nabla_x g_\lambda = 0.
$$
This explains why we study the evolution equation (\ref{eq:Bol1}) in the case of a non-constant restitution coefficient. In the constant case, we refer to \cite{MM2} for such an explanation.

\medskip

\subsection{Outline of the paper}
Section~\ref{sec:linearized-operator} is devoted to the study of the linearized problem.
In Section~\ref{sec:nonlinear-equation}, we go back to the nonlinear equation and we prove our main theorems concerning the Cauchy theory of our equation.
\medskip

\noindent\textbf{Acknowledgments.} We thank St\'{e}phane Mischler for fruitful discussions and his numerous comments and suggestions. 

\bigskip


\section{Properties of the linearized operator} 
\label{sec:linearized-operator}
\setcounter{equation}{0}
\setcounter{theo}{0}


\subsection{Notations and definitions}
For a given real number $a \in \R$, we define the half complex plane
$$
\Delta_a := \left\{ z \in \C, \, \Re e \, z > a \right\}.
$$

\smallskip
For some given Banach spaces $(E,\|\cdot \|_E)$ and $(\EE,\| \cdot
\|_\EE)$, we denote by $\mathscr{B}(E, \EE)$ the space of bounded linear
operators from $E$ to $\EE$ and we denote by $\| \cdot
\|_{\mathscr{B}(E,\EE)}$ or $\| \cdot \|_{E \to \EE}$ the associated operator norm. We write $\mathscr{B}(E) = \mathscr{B}(E,E)$ when $E=\EE$.
We denote by $\mathscr{C}(E,\EE)$ the space of closed unbounded linear
operators from $E$ to $\EE$ with dense domain, and $\mathscr{C}(E)=
\mathscr{C}(E,E)$ in the case $E=\EE$.

\smallskip
For a Banach space $X$ and   $\Lambda \in \mathscr{C}(X)$ we denote by $S_\Lambda(t)$, $t \ge
0$, its associated semigroup when it exists, by $\textrm{D}(\Lambda)$ its domain, by
$\textrm{N}(\Lambda)$ its null space and by $\mbox{R}(\Lambda)$ its range. 
We introduce the $\textrm{D}(\Lambda)$-norm defined as $\|f\|_{\textrm{D}(\Lambda)} = \|f\|_X + \|\Lambda f \|_X$ for $f \in \textrm{D}(\Lambda)$. More generally, for $k \in \N$, we define 
$$
\|f\|_{\textrm{D}(\Lambda^k)} = \sum_{j=0}^k \|\Lambda^j f\|_X, \quad  f \in \textrm{D}(\Lambda^k).
$$

We also denote by $\Sigma(\Lambda)$ its spectrum, so that for any $z$ belonging to the resolvent set $\rho(\Lambda) :=  \C
\backslash \Sigma(\Lambda)$, the operator $\Lambda - z$ is invertible and the resolvent operator
$$
\RR_\Lambda(z) := (\Lambda -z)^{-1}
$$
is well-defined, belongs to $\mathscr{B}(X)$ and has range equal to
$\textrm{D}(\Lambda)$.
We recall that $\xi \in \Sigma(\Lambda)$ is said to be an eigenvalue
if $N(\Lambda - \xi) \neq \{ 0 \}$. Moreover an eigenvalue $\xi \in
\Sigma(\Lambda)$ is said to be isolated if there exists $r>0$ such that
\[
\Sigma(\Lambda) \cap \left\{ z \in \C, \,\, |z - \xi| \le r \right\} =
\{ \xi \} .
\]
In the case when $\xi$ is an isolated eigenvalue we may define
$\Pi_{\Lambda,\xi} \in \mathscr{B}(X)$ the associated spectral projector by
$$
\Pi_{\Lambda,\xi} := - {1 \over
  2i\pi} \int_{ |z - \xi| = r' } (\Lambda-z)^{-1} \, dz
$$
with $0<r'<r$. Note that this definition is independent of the value
of $r'$ as the application
$
\C \setminus \Sigma(\Lambda) \to \mathscr{B}(X)$, $z \to \RR_{\Lambda}(z)$ is holomorphic.
For any $\xi \in \Sigma(\Lambda)$ isolated, it is well-known (see~\cite[Paragraph III-6.19]{Kato}) 
that $\Pi_{\Lambda,\xi}^2=\Pi_{\Lambda,\xi}$,  so that $\Pi_{\Lambda,\xi}$ is indeed a projector,
and that the ``associated projected semigroup"
\[
S_{\Lambda,\xi}(t) := -\frac{1}{2i \pi} \int_{|z-\xi|=r'} e^{zt}
\RR_\Lambda (z) \, dz , \quad t >0,
\]
satisfies
$$
\forall \, t>0, \quad S_{\Lambda,\xi}(t)
= \Pi_{\Lambda,\xi}  S_{\Lambda}(t) =
S_\Lambda(t)  \Pi_{\Lambda,\xi}.
$$
When moreover the so-called ``algebraic eigenspace" $\mbox{R}(\Pi_{\Lambda,\xi})$ is finite dimensional we say that
$\xi$ is a discrete eigenvalue, written as $\xi \in \Sigma_d(\Lambda)$. 



\smallskip
We shall need the following definition on
the convolution of semigroup (corresponding to composition at
the level of the resolvent operators). If one considers some Banach spaces $X_1$, $X_2$, $X_3$, for two given
  functions
  \[
  S_1 \in L^1(\R_+; \BBB(X_1,X_2)) \ \mbox{ and } \ S_2 \in L^1(\R_+;
  \BBB(X_2,X_3)),
  \]
the convolution $S_2 \ast S_1 \in L^1(\R_+; \BBB(X_1,X_3))$ is defined as
  $$
  \forall \, t \ge 0, \quad (S_2 * S_1)(t) := \int_0^t S_2(s) \, S_1 (t-s) \, ds.
  $$
%

\smallskip
Let us now introduce the notion of hypodissipative operators. If one consider a Banach space $(X, \| \cdot \|_X)$ and some operator $\Lambda \in \CCC(X)$, $(\Lambda - a)$ is said to be hypodissipative on~$X$ 
if there exists some norm $\Nt \cdot \Nt_X$ on $X$ equivalent to the initial norm $\| \cdot \|_X$ such that
$$
\forall \, f \in \textrm{D}(\Lambda), \quad  \exists \,  \phi \in F(f) \quad \text{s.t.} \quad \Re e \langle \phi , (\Lambda - a) f \rangle \leq 0, 
$$
where $\langle \cdot , \cdot \rangle$ is the duality bracket for the duality in $X$ and $X^*$ and $F(f) \subset X^*$
is the dual set of $f$ defined by
$$
F(f) = F_{\Nt \cdot \Nt_X} (f) := \left\{ \phi \in X^*, \, \langle \phi, f \rangle = \Nt f \Nt_X^2 =\Nt \phi \Nt_{X^*}^2 \right\}.
$$

\medskip

\subsection{Preliminaries on the steady states} \label{subsection:steadystates}
Let us recall results about the stationary equation 
\beqn \label{eq:stat1}
Q_{e_\lambda}(f,f) + \lambda^\gamma \, \Delta_v f = 0.
\eeqn
The main references for this subsection are~\cite{MM2} for the constant case and~\cite{AL3} for the non-constant case. We introduce the following notation: we shall say that a restitution coefficient $e(\cdot)$ satisfying Assumptions~\ref{ass} is belonging to the class $\E_m$ for some integer $m \geq1$ if $e(\cdot) \in \CC^m(0, \infty)$ and 
$$
\forall \, k=1,\dots, m, \quad \sup_{r \geq 0} r e^{(k)}(r) < \infty,
$$
where $e^{(k)}(\cdot)$ denotes the $k$-th order derivative of $e(\cdot)$.

\begin{rem}
For the physically relevant case of visco-elastic hard spheres, the restitution coefficient $e(\cdot)$ is given by~(\ref{eq:visco}) but admits also the following implicit representation (see~\cite{BP}):
$$
\forall \, r>0, \quad e(r) + a r^{\frac{1}{5}} e^{\frac{3}{5}}(r) = 1
$$
for some $a > 0$. Then, it is possible to deduce from such representation that $e(\cdot)$ belongs to the class $\E_m$ for any integer $m \geq 1$.
\end{rem}
In~\cite[Theorem~4.5]{AL3}, the authors state that if $e(\cdot)$ belongs to the class $\E_m$ for some integer $m \geq 4$, there exists $\lambda^\dagger \in (0,1]$ such that for any $\lambda \in [0,\lambda^\dagger)$, there exists a unique solution in $L^1_2$ of~(\ref{eq:stat1}) of mass $1$ and vanishing momentum. We denote $G_\lambda$ this solution. 

It is also proved in~\cite[Proposition~3.3]{AL3} that there exist $A>0$, $M>0$ such that for any $\lambda \in (0,\lambda^\dagger]$, $G_\lambda$ satisfies 
\beqn \label{eq:stat2}
\int_{\R^3} G_\lambda (v) \,  e^{A \, |v|^{3/2}} \, dv \leq M. 
\eeqn 

Let us point out that in the case of a constant coefficient, these results were already established. In~\cite[Theorem~1]{BGP} and~\cite[Theorem~5.2 \& Lemma~7.2]{GPV}, existence of solutions and regularity estimates are proved. In~\cite[Section~2.1]{MM2}, it is proved that these estimates are uniform in terms of the coefficient of inelasticity and in~\cite[Theorem~1.2]{MM2}, uniqueness of steady states is proved for a sufficiently small coefficient of inelasticity.  

\smallskip
Throughout the paper, we shall denote 
$$
m(v) = e^{b\langle v \rangle^\beta} \quad \text{with} \quad b>0, \quad \beta \in (0,1)
$$
the stretched exponential weight. We now state several lemmas on steady states $G_\lambda$ which are straightforward consequences of results from~\cite{MM2} and~\cite{AL3}. We shall use them several times in what follows. First, we recall a result of interpolation (see for example~\cite[~Lemma~B.1]{MM}) which is going to be very useful. 

\begin{lem} \label{lem:interp}
For any $k, \, q \in \N$, there exists $C>0$ such that for any $h \in H^{k'}_v \cap L^1_v(m^{12})$ with $k' = 8k+7(1+3/2)$
$$
\|h\|_{W^{k,1}_v(\langle v \rangle ^q m)} \leq C \|h \|_{H^{k'}_v}^{1/8} \|h\|_{L^1_v(m^{12})}^{1/8} \|h\|_{L^1_v(m)}^{3/4}.
$$
\end{lem}

Let us now prove estimate on Sobolev norm of $G_\lambda$. 
\begin{lem} \label{lem:Glambdabound}
Let $k$, $q \in \N$. We denote $k' = 8k+7(1+3/2)$. If $e(\cdot)$ belongs to the space $\E_{k'+1}$, then there exists $C>0$ such that
$$
\forall \, \lambda \in (0,\lambda^\dagger], \quad \|G_\lambda\|_{W^{k,1}_v(\langle v \rangle ^qm)} \leq C.
$$
\end{lem}

\begin{proof}
We deduce from~(\ref{eq:stat2}) that there exists $C>0$ such that for any $\lambda \in (0,\lambda^\dagger]$, $\|G_\lambda\|_{L^1_v(m)} \leq C$ and $\|G_\lambda\|_{L^1_v(m^{12})} \leq C$. We now use~\cite[Theorem~3.6]{AL3}, it gives us the following:
$$
\forall \, q \in \N, \, \, \forall\,  \ell \in [0,k'], \quad \sup_{\lambda \in (0, \lambda^\dagger]} \|G_\lambda\|_{H^\ell_v(\langle v \rangle^q)} < \infty.
$$
Gathering the previous estimates and using Lemma~\ref{lem:interp}, we obtain the result. Let us mention that in the case of a constant coefficient, we can prove this result using~\cite[~Proposition~2.1]{MM2}.
\end{proof}

We now give an estimate on the difference between $G_\lambda$ and $G_0$, the elastic equilibrium which is a Maxwelian distribution. 
\begin{lem} \label{lem:Glambda-G0}
Let $k \in \N$, $q \in \N$. We denote $k' = 8k+7(1+3/2)$. If $e(\cdot)$ belongs to the space~$\E_{k'+1}$, then there exists a function $\eps_1(\lambda)$ such that for any $\lambda \in (0, \lambda^\dagger]$
$$
\|G_\lambda - G_0\|_{W^{k,1}_v(\langle v \rangle^q m)} \leq \eps_1(\lambda) \quad \text{with} \quad \eps_1(\lambda) \xrightarrow[\lambda \rightarrow 0]{} 0 .
$$
\end{lem}

\begin{proof}
Theorem~4.1 from~\cite{AL3} implies that 
$$
\|G_\lambda - G_0\|_{H^{k'}_v} \xrightarrow[\lambda \rightarrow 0]{} 0.
$$
Using this estimate with Lemma~\ref{lem:interp} and Lemma~\ref{lem:Glambdabound}, it yields the result. We here precise that in the case of a constant coefficient, we can conclude using~\cite[Lemma~4.3]{MM2}. 
\end{proof}

\medskip
\subsection{The linearized operator and its splitting} \label{subsection:linearization}
\subsubsection{The collision operator $Q_{e_\lambda}$} 
The formula \eqref{eq:collop} suggests the natural splitting between gain and loss parts $Q_{e_\lambda} =Q_{e_\lambda}^+ - Q_{e_\lambda}^-$. The loss part $Q_{e_\lambda}^-$ can easily be defined in strong form noticing that
$$
\langle Q_{e_\lambda}^-(g,f), \psi \rangle = \int_{\R^3 \times \R^3 \times \Sp^2} g(v_*) f(v) \psi(v) |v-v_*| \, d\sigma \, dv_* \, dv  =: \langle f L(g), \psi \rangle,
$$
where $\langle \cdot, \cdot \rangle$ is the usual scalar product in $L^2$ and $L$ is the convolution operator 
$$
L(g)(v) = 4 \pi (| \cdot| \ast g)(v).
$$
In particular, we can notice that $L$ and $Q_{e_\lambda}^-$ are independent of the normal restitution coefficient. 
We also define the symmetrized (or polar form of the) bilinear collision operator $\widetilde{Q}_{e_\lambda}$ by setting
\begin{align} \label{eq:tildeQ}
\int_{\R^3}  \widetilde{Q}_{e_\lambda} (g,h) \psi \, dv = & \, \frac{1}{2}  \int_{\R^3 \times \R^3 \times \Sp^2}  g(v_*) h(v) |v-v_*|  \left[\psi(v') \, + \, \psi(v'_*) \right] \, d\sigma \, dv_* \, dv \\
&- \frac{1}{2}  \int_{\R^3 \times \R^3 \times \Sp^2}  g(v_*) h(v) |v-v_*|  \left[ \psi(v) \,+ \, \psi(v_*) \right] \, d\sigma \, dv_* \, dv . \nonumber
\end{align}
In other words, $\widetilde{Q}_{e_\lambda} (g,h) = (Q_{e_\lambda}(g,h) + Q_{e_\lambda}(h,g))/2$. The formula~(\ref{eq:tildeQ}) also suggests a splitting $\widetilde{Q}_{e_\lambda} = \widetilde{Q}_{e_\lambda}^+ - \widetilde{Q}_{e_\lambda}^-$ between gain and loss parts. We can notice that we have $\widetilde{Q}^+_{e_\lambda} (g,h) = (Q_{e_\lambda}^+(g,h) + Q_{e_\lambda}^+(h,g))/2$ and  $\widetilde{Q}_{e_\lambda}^- (g,h) = (Q_{e_\lambda}^-(g,h) + Q_{e_\lambda}^-(h,g))/2$.
In the elastic case ($\lambda=0$), we can easily define the collision operator in strong form using the pre-post collisional change of variables:
$$
Q_1(g,f)= \int_{\R^3 \times \Sp^2} \left[f(v') g(v'_*) - f(v) g(v_*) \right] |v-v_*| \, dv_* \, d\sigma.
$$
\smallskip

\subsubsection{Decomposition of the linearized operator}
We now study the equation $\partial_t h = \LL_\lambda h$ introduced in \eqref{eq:Bol2}  for $h=h(t,x,v), \, x \in \T^3, \, v \in \R^3$. We define the operator $\widehat{Q}_{e_\lambda}$ by 
$$
\widehat{Q}_{e_\lambda}(h) = Q_{e_\lambda}  (G_\lambda,h) + Q_{e_\lambda} (h, G_\lambda) = 2 \,  \widetilde{Q}_{e_\lambda} (h,G_\lambda),
$$
where $\widetilde{Q}_{e_\lambda}$ is defined in~(\ref{eq:tildeQ}). Using the weak formulation, we have
$$
\int_{\R^3}  \widehat{Q}_{e_\lambda} (h) \psi \, dv = \int_{\R^3 \times \R^3 \times \Sp^2}  G_\lambda(v) h(v_*) |v-v_*|  \left[\psi(v') \, + \, \psi(v'_*) \, - \, \psi(v) \, - \, \psi(v_*) \right] \, d\sigma \, dv_* \, dv 
$$
for any test function $\psi$.

Let us introduce the decomposition of the linearized operator $\LL_\lambda$. For any $\delta \in (0,1)$, we consider $\Theta_\delta =\Theta_\delta(v,v_*,\sigma) \in \CC^\infty$ bounded by one, which equals one on 
$$
\left\{|v| \le \delta^{-1} \text{ and } 2\delta \le |v-v_*| \le \delta^{-1} \text{ and } |\cos  \theta | \le 1- 2\delta \right\}
$$
and whose support is included in
$$
\left\{|v| \le 2\delta^{-1} \text{ and } \delta \le |v-v_*| \le 2\delta^{-1} \text{ and } |\cos  \theta | \le 1- \delta \right\}.
$$
We introduce the following splitting of the linearized elastic collisional operator $\widehat{Q}_1$ defined as $\widehat{Q}_1(h) =Q_1(G_0,h) + Q_1(h,G_0) $:
$$
\widehat{Q}_1=\widehat{Q}_{1,S}^{+,\ast} + \widehat{Q}_{1,R}^{+,\ast} - L(G_0)
$$
with the truncated operator
$$
\widehat{Q}_{1,S}^{+,\ast} (h) =   \int_{\R^3 \times \Sp^2}  \Theta_\delta \! \left[G_0(v'_*) h(v') \, + \, G_0(v') h(v'_*) - G_0(v) h(v_*) \right] \, |v-v_*| \, dv_* \, d\sigma,
$$
the corresponding remainder operator
$$
\widehat{Q}_{1,R}^{+,\ast} (h) =   \int_{\R^3 \times \Sp^2} (1-\Theta_\delta) \! \left[G_0(v'_*) h(v') \, + \, G_0(v') h(v'_*) - G_0(v) h(v_*) \right] \, |v-v_*| \, dv_* \, d\sigma
$$
and 
$$
L(G_0) = 4\pi \,  \left(G_0 \ast | \cdot | \right).
$$

We can then write a decomposition for the full linearized operator $\LL_\lambda$:
\bean \label{decomp}
\LL_\lambda h
&=& \widehat{Q}_{e_\lambda}(h) -\widehat{Q}_{1}(h) + \widehat{Q}_1(h)  +  \lambda^\gamma \, \Delta_v h - v \cdot  \nabla_x h
\\
&=&  \widehat{Q}_{e_\lambda}(h) - \widehat{Q}_1(h) + \widehat{Q}^{*,+}_{1,S}(h) + \widehat{Q}^{+,*}_{1,R}(h)  - L(G_0) \, h +  \lambda^\gamma \, \Delta_v h - v \cdot  \nabla_x h.
\eean
We denote 
$$
\AA_\delta h :=  \widehat{Q}^{*,+}_{1,S}(h)
$$
and
$$
\BB_{\lambda, \delta} h :=  \widehat{Q}_{e_\lambda}(h) - \widehat{Q}_1(h) + \widehat{Q}^{+,*}_{1,R}(h) +  \lambda^\gamma \, \Delta_v h - v \cdot  \nabla_x h -  L(G_0) \, h.
$$

Thanks to the truncation, we can use the so-called Carleman representation (see~\cite[~Chapter~1, Section~4.4]{Vill}) and write the truncated operator $\AA_\delta$
as an integral operator
\beqn \label{eq:kernelA}
\AA_\delta(h)(v) = \int_{\R^3} k_\delta(v, v_*) \, h(v_*) \, dv_*
\eeqn
for some smooth kernel $k_\delta \in C^\infty_c \left( \R^3 \times \R^3 \right)$.

We also introduce the collision frequency $\nu := L(G_0)$ which satisfies  $\nu(v) \approx \langle v \rangle$ i.e. there exist some constants $\nu_0$, $\nu_1 >0$ such that:
\beqn \label{nu}
\forall \, v \in \R^3, \quad 0<\nu_0 \leq \nu_0 \langle v\rangle \le \nu(v) \le \nu_1\langle v \rangle.
\eeqn

\smallskip
\subsubsection{Spaces at stake}

Let us consider the three Banach spaces 
\begin{align*}
&E_1:= W_x^{s+2,1} W_v^{4, 1} (\langle v \rangle^2 m), \\
&E_0 = E := W_x^{s,1} W^{2,1}_v (\langle v \rangle m), \\
&E_{-1}: = W_x^{s-1,1} L^1_v (m), \\
&\EE := W_x^{s,1} L^1_v (m)
\end{align*} 
for some $s \in \N^*$.

\smallskip
In the remaining part of the paper, we suppose that in addition to Assumptions \ref{ass}, in the non constant case, the following assumption on $e(\cdot)$ holds:
\begin{ass1} \label{ass2}
The coefficient of restitution $e(\cdot)$ belongs to $\E_{k^\dagger+1}$ where \linebreak$k^\dagger:= 32 + 7(1+3/2)$.
\end{ass1}
It allows us to get uniform bounds on the $E_j$-norms of $G_\lambda$ and uniform estimates on the $E_j$-norms of the difference $G_\lambda - G_0$ for $j=-1,0,1$ (thanks to Lemmas~\ref{lem:Glambdabound} and~\ref{lem:Glambda-G0}). 

The operator $\LL_\lambda$ is bounded from $E_j$ to $E_{j-1}$ for $j=0,1$. The operators $\Delta_v$ and $v \cdot \nabla_x$ are clearly bounded from $E_j$ to $E_{j-1}$. As far as $\widehat{Q}_{e_\lambda}$ is concerned, we are going to use the result of interpolation Lemma~\ref{lem:interp}. 
\begin{lem}
Let us consider $k, \, q \in \N$. We denote $k' = 8k+7(1+3/2)$. If $e(\cdot)$ belongs to the space~$\E_{k'+1}$, then $\widehat{Q}_{e_\lambda}$ is bounded from $W^{s,1}_x W^{k,1}_v (\langle v \rangle ^{q+1}m)$ to $W^{s,1}_x W^{k,1}_v (\langle v \rangle ^{q}m)$.
\end{lem}

 \begin{proof}
 As far as the case of a constant coefficient is concerned, Proposition 3.1 from~\cite{MM} gives us 
 $$
 \|\widehat{Q}_{e_\lambda}(h) \|_{L^1_v(\langle v \rangle^q m)} \leq C \|G_\lambda\|_{L^1_v(\langle v \rangle^{q+1} m)} \|h\|_{L^1_v(\langle v \rangle^{q+1} m)} \leq C  \|h\|_{L^1_v(\langle v \rangle^{q+1} m)},
 $$
 where the last inequality comes from Lemma~\ref{lem:Glambdabound}. Concerning the case of a non-constant coefficient, we use both Lemma~\ref{lem:Glambdabound} and~\cite[Theorem~1]{ACG} and we get:
 $$
  \|\widehat{Q}_{e_\lambda}(h) \|_{L^1_v(\langle v \rangle^q m)} \leq C \|G_\lambda\|_{L^1_v(\langle v \rangle^{q+1} m)} \|h\|_{L^1_v(\langle v \rangle^{q+1} m)} \leq C  \|h\|_{L^1_v(\langle v \rangle^{q+1} m)}.
$$
The $x$-derivatives commute with the operator $\widehat{Q}_{e_\lambda}$, therefore we can do the proof with $s=0$ without loss of generality. We first look at the case $L^1_xL^1_v(\langle v \rangle^q m)$ before treating the $v$-derivatives. Using Fubini theorem and the previous inequalities, we obtain
$$
\|\widehat{Q}_{e_\lambda} h \|_{L^1_xL^1_v(\langle v \rangle^q m)} \leq C \|h\|_{L^1_xL^1_v(\langle v \rangle^{q+1} m)}.
$$
We now treat the case $L^1_xW^{1,1}_v(\langle v \rangle^q m)$. We use the property
\beqn \label{derivQ}
\partial_v Q_{e_\lambda}^\pm (f,g) = Q_{e_\lambda}^\pm(\partial_v f,g) + Q_{e_\lambda}^\pm(f, \partial_v g).
\eeqn
We then compute
$$
\partial_v \widehat{Q}_{e_\lambda} h = Q_{e_\lambda}(\partial_v G_\lambda, h) + Q_{e_\lambda} (G_\lambda, \partial_v h) + Q_{e_\lambda} (\partial_v h, G_\lambda) + Q_{e_\lambda} (h, \partial_v G_\lambda).
$$
Using Lemma~\ref{lem:Glambdabound},~\cite[Proposition~3.1]{MM} in the constant case and~\cite[Theorem~1]{ACG} in the non-constant case, the $L^1_v(\langle v \rangle^q m)$-norm of each term can be bounded by $C \|h\|_{W^{1,1}_v(\langle v \rangle^{q+1} m)}$. Again using Fubini theorem, we deduce that
$$
\|\partial_v \widehat{Q}_{e_\lambda} h\|_{L^1_x L^1_v(\langle v \rangle^q m)} \leq C \|h \|_{L^1_x W^{1,1}_v(\langle v \rangle^{q+1} m)}. 
$$
The higher-order terms are dealt with in a similar manner, which concludes the proof. 
\end{proof}

Under the assumptions made on $e(\cdot)$, using the previous lemma, we can conclude that $\widehat{Q}_{e_\lambda}$ is bounded from $E_j$ to $E_{j-1}$ for $j=0,1$.
\medskip

\medskip

\subsection{Hypodissipativity of $\BB_{\lambda,\delta}$ and boundedness of $\AA_\delta$}
\begin{lem} \label{lem:dissip1}
Let us consider $k \geq0$, $s \ge k$ and $q \geq 0$. We denote $k' = 8k+7(1+3/2)$. If $e(\cdot)$ belongs to the space~$\E_{k'+1}$, then there exist $\lambda_1 \in (0,\lambda^\dagger)$, $\delta>0$ and $\alpha_0>0$ such that for any $\lambda \in [0,\lambda_1]$, $\BB_{\lambda,\delta} +\alpha_0$ is hypodissipative in $W^{s,1}_x W^{k,1}_v(\langle v\rangle ^q m)$. 
\end{lem}

\begin{proof}
Observe first that the $x$-derivatives commute with the operator $\BB_{\lambda,\delta}$, therefore we can do the proof for $s=0$ without loss of generality. We first look at the case $L^1_x L^1_v( \langle v \rangle^q m)$ before treating the $v$-derivatives. We compute 
$$
\begin{aligned}
&\quad \int_{\R^3\times\T^3}\BB_{\lambda,\delta} (h) \, \text{sign} (h) \, dx \, \langle v \rangle^q m(v) \, dv \\
&=   \int_{\R^3\times\T^3} (\widehat{Q}_{e_\lambda} (h) - \widehat{Q}_1(h) ) \, \text{sign} (h) \, dx \, \langle v \rangle^q m(v) \, dv \\
&\quad +\int_{\R^3\times\T^3} \widehat{Q}_{1,R}^{+,*}(h) \, \text{sign} (h) \, dx \, \langle v \rangle^q m(v) \, dv \\
&\quad + \lambda^\gamma    \int_{\R^3\times\T^3} \Delta_v h \, \text{sign} (h) \, dx \, \langle v \rangle^q m(v) \, dv \\
&\quad - \int_{\R^3\times\T^3} v \cdot \nabla_x h \, \text{sign} (h) \, dx \, \langle v \rangle^q m(v) \, dv \\
&\quad- \int_{\R^3\times\T^3} \nu \, h \, \text{sign} (h) \, dx \, \langle v \rangle^q m(v) \, dv \\
&=:  I_1(h) + I_2(h) + I_3(h) + I_4(h) + I_5(h).
\end{aligned}
$$

We first deal with $I_1$ splitting the difference $\widehat{Q}_{e_\lambda} - \widehat{Q}_1$ into several parts and using that $Q_{e_\lambda}^- = Q_1^-$:
\begin{align*}
\widehat{Q}_{e_\lambda} h -\widehat{Q}_1 h = & \, Q_{e_\lambda}^+ (h,G_\lambda) -  Q_1^+ (h,G_\lambda) +  Q_1^+ (h,G_\lambda-G_0) \\
&+  Q_{e_\lambda}^+ (G_\lambda,h) -  Q_1^+ (G_\lambda,h) +  Q_1^+ (G_\lambda-G_0,h) \\
&- Q_1^-(h,G_\lambda - G_0) - Q_1^-(G_\lambda - G_0,h) \\
=& \, 2 \left[\widetilde{Q}_{e_\lambda}^+ (h, G_\lambda) - \widetilde{Q}_1^+(h,G_\lambda) + \widetilde{Q}_1^+(h, G_\lambda - G_0) - \widetilde{Q}_1^-(h,G_\lambda - G_0) \right]. 
\end{align*}
We now use a result given by~\cite[Proposition~3.1]{MM} which can be easily extended to others weights of type $\langle v \rangle^q m$. We can treat together the terms $\widetilde{Q}_1^+ (h,G_\lambda-G_0)$ and \linebreak$\widetilde{Q}_1^-(h, G_\lambda - G_0)$. Because of~\cite[Proposition 3.1]{MM}, their $L^1_v(\langle v \rangle^q m)$-norm are bounded from above by $C \, \|G_\lambda -G_0\|_{L^1_v(\langle v \rangle^{q+1} m)} \|h\|_{L^1_v(\langle v \rangle^{q+1} m)}.$
Then, using Lemma~\ref{lem:Glambda-G0}, we obtain
\bear \label{eq:epsilon11}
\|\widetilde{Q}_1^{\pm} (h,G_\lambda-G_0)\|_{L^1_v( \langle v \rangle^q m)} \leq   C \, \eps_1(\lambda) \|h\|_{L^1_v(\langle v \rangle^{q+1} m)}, \quad \eps_1(\lambda) \xrightarrow[\lambda \rightarrow 0]{} 0.
\eear
Concerning the term $\widetilde{Q}_{e_\lambda}^+ (h_t,G_\lambda) -  \widetilde{Q}_1^+ (h_t,G_\lambda)$, we use~\cite[Theorem~3.11]{AL3} (we can use~\cite[Proposition~3.2]{MM} for the constant case) and Lemma~\ref{lem:Glambdabound}. It gives us that there exists $\lambda_1 \in (0,\lambda^\dagger]$ such that for any $\lambda \in (0,\lambda_1)$:
\begin{align}
\| \widetilde{Q}_{e_\lambda}^+ (h,G_\lambda) -  \widetilde{Q}_1^+ (h,G_\lambda) \|_{L^1_v(\langle v \rangle^q m)}  &\leq C \lambda^{\frac{\gamma}{8+3\gamma}} \, \|G_\lambda \|_{W^{1,1}_v ( \langle v \rangle^{q+1} m)} \, \|h\|_{L^1_v(\langle v \rangle^{q+1} m)} \nonumber \\
&\leq C \eps_2(\lambda) \,\|h\|_{L^1_v(\langle v \rangle^{q+1} m)}, \quad \eps_2(\lambda) \xrightarrow[\lambda \rightarrow 0]{} 0. \label{eq:epsilon21}
\end{align}
In \cite{AL3} and \cite{MM}, the results are only stated in the case $q=0$ but it is easy to extend these results using the fact that $\langle v'\rangle^q \leq C \, \langle v \rangle^q \, \langle v_* \rangle^q$.
Gathering (\ref{eq:epsilon11}) and (\ref{eq:epsilon21}), we thus obtain
\beqn \label{eq:term1}
I_1(h) \le \int_{\R^3\times \T^3} \left|\widehat{Q}_{e_\lambda} (h) - \widehat{Q}_1 (h)\right|  \, dx \, \langle v \rangle^q m(v) \, dv \leq \eps(\lambda) \|h\|_{L^1_x L^1_v(\langle v \rangle^{q+1} m)}, \quad \eps(\lambda) \xrightarrow[\lambda \rightarrow 0]{} 0.
\eeqn

As far as $I_2$ is concerned, we first recall that~\cite[Proposition~2.1]{Mou} establishes that there holds
$$
\forall \, h \in L^1_v(\langle v \rangle m), \quad \|\widehat{Q}_{1,R}^{+,*}(h)\|_{L^1_v(m)} \leq \Lambda(\delta) \|h\|_{L^1_v(\langle v \rangle m)}, \quad  \Lambda(\delta) \xrightarrow[\delta \rightarrow 0]{} 0,
$$
where however the definition of $\Theta_\delta$ is slightly different and only the case $q=0$ is treated. But it is straightforward to extend the proof to the present situation. We hence have
\beqn \label{eq:term2}
I_2(h) \leq \int_{\R^3 \times \T^3} |\widehat{Q}_{1,R}^{+,*}(h_t)| \, dx \, \langle v \rangle^q m(v) \, dv \leq \Lambda(\delta) \|h\|_{L^1_x L^1_v(\langle v \rangle^{q+1} m)}, \quad \Lambda(\delta) \xrightarrow[\delta \rightarrow 0]{} 0.
\eeqn

Concerning the term with the Laplacian, we write performing two integrations by parts 
\begin{align*}
\int_{\R^3\times\T^3} \Delta_v h_t \,  \text{sign} (h_t)  \, \langle v \rangle^q m \, dv \, dx
=& -\int_{\T^3\times\R^3} \left|\nabla_v h \right|^2 \, \text{sign}'(h)  \, \langle v \rangle^q m \, dv \, dx \\
&- \int_{\T^3\times\R^3} \nabla_v h \,  \, \text{sign}(h) \cdot \nabla_v (\langle v \rangle^q m(v)) \, dv \,  dx \\
\le&  \int_{\T^3\times\R^3} \nabla_v |h | \cdot \nabla_v(\langle v \rangle^q m) \, dv \, dx \\
=&  \int_{\T^3\times\R^3} |h| \, \Delta_v (\langle v \rangle^q m)  \, dv \, dx \\
=& \int_{\R^3\times\T^3}  |h| \, \langle v \rangle^q m \, \frac{\Delta_v (\langle v \rangle^q m)}{\langle v \rangle^q m} \,dx \, dv.
\end{align*}
Since $\Delta_v (\langle v \rangle^q m) / (\langle v \rangle^q m)$ is bounded in $\R^3$, we can write 
\beqn \label{eq:term3}
I_3(h) \leq C \, \lambda^\gamma \, \|h\|_{L^1_x L^1_v(\langle v \rangle^q m)} \leq C \, \lambda^\gamma \, \|h\|_{L^1_x L^1_v(\langle v \rangle^{q+1} m)}.
\eeqn

We notice that 
\beqn \label{eq:term4}
I_4(h) =0
\eeqn  
because the term $v \cdot \nabla_x h$ has a divergence structure in~$x$.

Finally, let us deal with $I_5$. We use property~(\ref{nu}), more precisely the fact that $\nu(v)$ is bounded below by $\nu_0 \langle v \rangle$:
\beqn \label{eq:term5}
I_5(h) = - \int_{\R^3 \times \T^3} |h| \,  dx \, \nu \, \langle v \rangle^q m(v) \, dv \leq - \nu_0 \, \|h\|_{L^1_x L^1_v (\langle v \rangle^{q+1} m)}.
\eeqn

Gathering~(\ref{eq:term1}), (\ref{eq:term2}), (\ref{eq:term3}), (\ref{eq:term4}) and~(\ref{eq:term5}), we obtain that for any $\lambda \in (0,\lambda_1)$
$$
\int_{\R^3\times\T^3} \BB_{\lambda, \delta} h \, \text{sign} (h) \, dx \, \langle v \rangle^q m(v) \, dv
\leq (\Lambda(\delta) + \eps(\lambda) + C \lambda^\gamma - \nu_0) \|h\|_{L^1_x L^1_v(\langle v \rangle^{q+1} m)}. 
$$
Up to make decrease the value of $\lambda_1$, we can suppose that for any $\lambda \in [0, \lambda_1]$, we have $\eps(\lambda) + C\lambda^\gamma < \nu_0$. Then, we choose $\delta$ close enough to $0$ in order to have 
\beqn \label{eq:alpha0}
\alpha_0 := - \left( \Lambda(\delta) + \max_{\lambda \in [0,\lambda_1]} \left[ \eps(\lambda) + C\lambda^\gamma \right] - \nu_0\right) > 0.
\eeqn
We hence have 
$$
\int_{\R^3\times\T^3} \BB_{\lambda, \delta} h \, \text{sign} (h) \, dx \, \langle v \rangle^q m(v) \, dv
\leq - \alpha_0 \, \|h\|_{L^1_x L^1_v(\langle v \rangle^{q+1}  m)}. 
$$
In particular, we deduce that for any $\lambda \in [0,\lambda_1]$, $\BB_{\lambda, \delta} + \alpha_0$ is dissipative in $L^1_x L^1_v(\langle v \rangle^q m)$. 

Let us now treat the $v$-derivatives. We are going to deal with the case $W^{1,1}_x W^{1,1}_v(\langle v \rangle^q m)$, the higher-order cases are similar. 
Thanks to~(\ref{derivQ}), we compute:
\begin{align*}
\partial_v (\BB_{\lambda,\delta} h) =& \partial_v\left(\widehat{Q}^{+,*}_{1,R} (h) - \nu h \right) + \partial_v \left((\widehat{Q}_{e_\lambda} - \widehat{Q}_1) (h) \right) + \lambda^\gamma \Delta_v \partial_v h - \partial_x h - v \cdot \nabla_x \partial_v h.  
\end{align*}

Let us treat the first term: 
$$ 
\partial_v\left(\widehat{Q}^{+,*}_{1,R} (h) - \nu h \right) = \widehat{Q}^{+,*}_{1,R} (\partial_v h) - \nu \, \partial_v h + \RR h
$$
with
$$
\RR h := Q_1(h, \partial_v G_0) + Q_1(\partial_v G_0, h) - (\partial_v \AA_\delta)(h) + \AA_\delta (\partial_v h).
$$
Using the form (\ref{eq:kernelA}) of the operator $\AA_\delta$ and performing one integration by part, we can show that
$$
\| (\partial_v \AA_\delta)(h) \|_{{L^1_x L^1_v(\langle v \rangle^q m)}} + \|\AA_\delta (\partial_v h)\|_{{L^1_x L^1_v(\langle v \rangle^q m)}} \leq C_\delta \, \|h\|_{L^1_x L^1_v (\langle v \rangle^{q} m)}.
$$
Combining this inequality with estimates~\cite[Proposition~3.1]{MM} on the elastic bilinear operator $Q_1$ of, we obtain
$$
\| \RR h \|_{L^1_x L^1_v(\langle v \rangle^q m)} \leq \, C_\delta \, \|h\|_{L^1_x L^1_v (\langle v \rangle^{q+1} m)}
$$
for some constant $C_\delta >0$. 

Let us now deal with the second term coming from the difference $\widehat{Q}_{e_\lambda} - \widehat{Q}_1$:
\begin{align*}
\partial_v \left((\widehat{Q}_{e_\lambda} - \widehat{Q}_1) h \right) =& \, (\widehat{Q}_{e_\lambda} - \widehat{Q}_1) ( \partial_v h) \\
&+ 2 \left[ \widetilde{Q}_{e_\lambda}^+(h, \partial_v G_\lambda) - \widetilde{Q}_1^+(h, \partial_v G_\lambda) \right] \\
&+2 \left[\widetilde{Q}_1^+(h, \partial_v (G_\lambda-G_0)) - \widetilde{Q}_1^-(h, \partial_v (G_\lambda-G_0)) \right].
\end{align*}
Arguing as before, we obtain 
$$
\| \widetilde{Q}_{e_\lambda}^+(h, \partial_v G_\lambda) - \widetilde{Q}_1^+(h, \partial_v G_\lambda) \|_{L^1_x L^1_v ( \langle v \rangle ^q m)} \leq \eps(\lambda) \| h \|_{L^1_x L^1_v ( \langle v \rangle ^{q+1} m)} 
$$
and
$$
\| \widetilde{Q}_1^+(h, \partial_v (G_\lambda-G_0)) - \widetilde{Q}_1^-(h, \partial_v (G_\lambda-G_0)) \|_{L^1_x L^1_v ( \langle v \rangle ^q m)} \leq \eps(\lambda) \| h \|_{L^1_x L^1_v ( \langle v \rangle ^{q+1} m)} 
$$
with $\eps(\lambda) \xrightarrow[\lambda \rightarrow 0]{} 0$.

All together, we deduce that 
$$
\partial_v (\BB_{\lambda,\delta} h)  = \BB_{\lambda, \delta} \partial_v h + \RR'(h)
$$
with 
$$
\| \RR' h \|_{L^1_x L^1_v (\langle v \rangle^q m)} \leq C_\delta \, \| h \|_{L^1_x L^1_v (\langle v \rangle^{q+1} m)} + \eps(\lambda) \,  \| h \|_{L^1_x L^1_v (\langle v \rangle^{q+1} m)}  + \| \nabla_x h \|_{L^1_x L^1_v (\langle v \rangle^q m)}.
$$

We now use the proof of the previous case to finally deduce the following estimate:
\begin{align*}
&\quad \int_{\R^3\times \T^3} \partial_v (\BB_{\lambda,\delta} h) \, \text{sign} (\partial_v h) \,dx \, \langle v \rangle^q m \, dv\\
&\leq\,  -\alpha_0  \| \nabla_v h \|_{L^1_x L^1_v ( \langle v \rangle^{q+1} m)} +  C_\delta \, \| h \|_{L^1_x L^1_v (\langle v \rangle^{q+1} m)} \\
&\quad + \eps(\lambda) \,  \| h \|_{L^1_x L^1_v (\langle v \rangle^{q+1} m)}  + \| \nabla_x h \|_{L^1_x L^1_v (\langle v \rangle^q m)},
\end{align*}
where $\alpha_0$ is defined in~(\ref{eq:alpha0}). 

Again using the proof of the previous case, we also have:
\begin{align*}
&\int_{\R^3\times \T^3} \BB_{\lambda,\delta} (h) \, \text{sign} (h) \,dx \, \langle v \rangle^q m \, dv + \int_{\R^3\times \T^3} \partial_x (\BB_{\lambda,\delta} h) \, \text{sign} (\partial_x h) \,dx \, \langle v \rangle^q m \, dv \\
\leq \,&- \alpha_0 \left(\|h\|_{L^1_x L^1_v(\langle v \rangle^{q+1} m)} + \|\nabla_x h\|_{L^1_x L^1_v(\langle v \rangle^{q+1}  m)} \right).
\end{align*}

We now introduce the norm 
$$
\|h\|_* := \|h\|_{L^1_xL^1_v(\langle v \rangle^q m)} + \|\nabla_x h\|_{L^1_xL^1_v(\langle v \rangle^q m)} + \eta \|\nabla_v h\|_{L^1_xL^1_v(\langle v \rangle^q m)} 
$$
for some $\eta >0$ (to be fixed later) which is equivalent to the classical $W^{1,1}_x W^{1,1}_v ( \langle v \rangle^q m)$-norm. We deduce
\begin{align*}
&\quad \int_{\R^3\times \T^3} \BB_{\lambda,\delta} (h) \, \text{sign} (h) \,dx \, \langle v \rangle^q m \, dv + \int_{\R^3\times \T^3} \partial_x (\BB_{\lambda,\delta} h) \, \text{sign} (\partial_x h) \,dx \, \langle v \rangle^q m \, dv \\
&\qquad \qquad \qquad + \eta \,
\int_{\R^3\times \T^3} \partial_v (\BB_{\lambda,\delta} h) \, \text{sign} (\partial_v h) \,dx \, \langle v \rangle^q m \, dv \\
 &\leq - \alpha_0
\left( \|h\|_{L^1_xL^1_v(\langle v \rangle^{q+1}  m)} + \|\nabla_x h\|_{L^1_xL^1_v(\langle v \rangle^{q+1} m)} + \eta \|\nabla_v h\|_{L^1_xL^1_v(\langle v \rangle^{q+1} m)} \right) \\
&+ \eta \left(C_\delta \, \| h \|_{L^1_x L^1_v (\langle v \rangle^{q+1} m)} + \eps(\lambda) \,  \| h \|_{L^1_x L^1_v (\langle v \rangle^{q+1} m)}  + \| \nabla_x h \|_{L^1_x L^1_v (\langle v \rangle^{q+1} m)}\right) \\
&\leq  (-\alpha_0 + o(\eta)) \left( \|h\|_{L^1_xL^1_v(\langle v \rangle^{q+1}  m)} + \|\nabla_x h\|_{L^1_xL^1_v(\langle v \rangle^{q+1} m)} + \eta \|\nabla_v h\|_{L^1_xL^1_v(\langle v \rangle^{q+1} m)} \right)
\end{align*}
with $o(\eta)\xrightarrow[\eta \rightarrow 0]{} 0$. We choose $\eta$ close enough to $0$ so that $\alpha_1 := \alpha_0 - o(\eta) > 0$. We thus obtain that $\BB_{\lambda,\delta}+\alpha_1$ with $\alpha_1>0$ is dissipative in $W^{1,1}_xW^{1,1}_v(\langle v \rangle^q m)$ for the norm $\| \cdot \|_*$ and thus hypodissipative in $W^{1,1}_xW^{1,1}_v(\langle v \rangle^q m)$.
\end{proof}

Let us clarify what implies the previous lemma in the following result:
\begin{lem} \label{lem:dissip}
Under the assumptions \ref{ass} and \ref{ass2} made on $e(\cdot)$, there exist $\lambda_1 \in (0,\lambda^\dagger]$, $\delta>0$ and $\alpha_0>0$ such that for any $\lambda \in [0,\lambda_1]$, $\BB_{\lambda,\delta} + \alpha_0$ is hypodissipative in $E_j$, $j=-1,0,1$ and $\EE$.
\end{lem}

The boundedness of $\AA_\delta$ is treated in~\cite{GMM}. Let us recall Lemma 4.16 of~\cite{GMM}.

\begin{lem} \label{lem:Adelta}
For any $s \in \N$, the operator $\AA_\delta$ maps $L^1_v(\langle v \rangle)$ into $H^s_v$ functions with compact support, with explicit bounds (depending on $\delta$) on the $L^1_v(\langle v \rangle) \rightarrow H^s_v$ norm and on the size of the support.

More precisely, there are two constants $C_{s, \delta}$ and $R_\delta$ such that for any $h \in L^1_v(\langle v \rangle )$ 
$$
K:= \mathrm{supp} \, \AA_\delta h \subset B(0, R_\delta), \quad \|\AA_\delta h \|_{H^s_v(K)} \leq C_{s, \delta} \| h \|_{L^1_v(\langle v \rangle)}.
$$
In particular, we deduce that $\AA_\delta$ is in $\BBB(E_j)$ for $j=-1,0,1$ and $\AA_\delta$ is in $\BBB(\EE,E)$.
\end{lem}

\medskip

\subsection{Regularization properties of $T_n:=\left( \AA_\delta \, S_{\BB_{\lambda,\delta}} \right)^{(*n)}$}
Let us consider $\lambda_1$ and $\alpha_0$ provided by Lemma~\ref{lem:dissip}. 
\begin{lem} \label{lem:Tn}
Let $\lambda$ be in $(0, \lambda_1]$. The time indexed family $T_n$ of operators satisfies the following: for any $\alpha_0' \in (0,\alpha_0)$, there are some constructive constants $C_\delta>0$ and $R_\delta$ such that for any $t \geq 0$
$$
\text{supp} \, T_n(t)h \subset K := B(0, R_\delta),
$$
and
\begin{align}
&\| T_1(t) h\|_{W^{s+1,1}_{x,v} (K)} \leq C \frac{e^{-\alpha_0't}}{t} \|h\|_{W^{s,1}_{x,v}( \langle v \rangle m)}, \quad \text{if} \, \, s \geq 1; \label{eq:T_1}\\
&\| T_2(t) h\|_{W^{s+1/2,1}_{x,v} (K)} \leq C e^{-\alpha_0't} \|h\|_{W^{s,1}_{x,v}( \langle v \rangle m)}, \quad \text{if} \, \, s \geq 0. 
\end{align}
\end{lem}

\begin{proof}
We first consider $h_0 \in W^{s,1}_{x,v}(\langle v \rangle m)$, $s \in \N$. Using Lemma~\ref{lem:Adelta} and the fact that the \linebreak $x$-derivatives commute with both $\AA_\delta$ and $\BB_{\lambda,\delta}$ and thus with $T_1(t)$, we get
$$
\|T_1(t) h_0\|_{W^{s,1}_x W^{s+1,1}_v(K)} = \| \AA_\delta \, S_{\BB_{\lambda,\delta}}(t) h_0 \|_{W^{s,1}_x W^{s+1,1}_v(K)}
\leq C \, \|S_{\BB_{\lambda,\delta}}(t) h_0 \|_{W^{s,1}_{x,v}(K)}.
$$
We then use that $\BB_{\lambda, \delta} + \alpha_0$ is dissipative in $W^{s,1}_{x,v}(\langle v \rangle m)$ (Lemma~\ref{lem:dissip}) to obtain
\beqn \label{eq:T_1'}
\|T_1(t) h_0\|_{W^{s,1}_x W^{s+1,1}_v(K)} \leq C e^{-\alpha_0 t} \| h_0\|_{W^{s,1}_{x,v} ( \langle v \rangle m)}.
\eeqn

Assume now $h_0 \in W^{s,1}_x W^{s+1,1}_v (\langle v \rangle m)$ and consider $g_t = e^{\BB_{\lambda, \delta}t} ( \partial^\beta_x h_0)$, for any $|\beta| \leq s$, which satisfies (using the fact that the $x$-derivatives commute with the semigroup)
$$
\partial_t g_t + v \cdot \nabla_x g_t = 
Q_{e_\lambda} (G_\lambda, g_t) + Q_{e_\lambda}(g_t, G_\lambda) + \lambda^\gamma \Delta_v g_t - \AA_\delta g_t.
$$
Let us define $D_t:= t \nabla_x + \nabla_v$. $D_t$ commute with the free transport equation and the Laplacian $\Delta_v$. Using these properties of commutativity and the property~(\ref{derivQ}) of the collision operator, we have
\begin{align*}
\partial_t (D_t g_t) + v \cdot \nabla_x(D_t g_t) = &\, Q_{e_\lambda}(\nabla_vG_\lambda, g_t) + Q_{e_\lambda}( g_t, \nabla_v G_\lambda) + Q_{e_\lambda}(G_\lambda, D_t g_t) \\
& \, +Q_{e_\lambda}(D_t g_t, G_\lambda) + \lambda^\gamma \Delta_v g_t - D_t(\AA_\delta g_t).
\end{align*}
With the notations of~(\ref{eq:kernelA}), we rewrite the last term as
\begin{align*}
D_t(\AA_\delta g_t)(v) =& \, D_t \int_{\R^3} k_\delta(v,v_*) \, g_t(v_*) \, dv_* \\
=&  \int_{\R^3} \nabla_v k_\delta(v,v_*) \, g_t(v_*) \, dv_* -  \int_{\R^3} k_\delta(v,v_*) \, \nabla_{v_*} g_t(v_*) \, dv_* \\
&+  \int_{\R^3} k_\delta(v,v_*) \, (D_t g_t)(v_*) \, dv_* \\
=& \, \AA^1_\delta g_t + \AA^2_\delta g_t + \AA_\delta (D_t g_t),
\end{align*}
where $\AA^1_\delta$ stands for the integral operator associated to the kernel $\nabla_v k_\delta$ and $\AA^2_\delta$ stands for the integral operator associated to the kernel $\nabla_{v_*} k_\delta$. All together, we may write
$$
\partial_t(D_t g_t) = \BB_{\lambda, \delta} (D_t g_t) + \II_\delta(g_t)
$$
with 
$$
\II_\delta f = Q_{e_\lambda} ( \nabla_v G_\lambda, f) + Q_{e_\lambda} (f, \nabla_v G_\lambda) - \AA^1_\delta f - \AA^2_\delta f,
$$
which satisfies
$$
\| \II_\delta f \|_{L^1_v(\langle v \rangle m)} \leq C_\delta \|f\|_{L^1_v(\langle v\rangle^2 m)}.
$$

Then arguing as in the proof of Lemma~\ref{lem:dissip}, we obtain, for any $\alpha_0'' \in (0,\alpha_0)$ and for $\eta$ small enough
$$
\frac{d}{dt} \left( e^{\alpha_0'' t} \int_{\R^3 \times \T^3} (\eta |D_t g_t| + |g_t|) \langle v \rangle m \, dx \, dv \right) \leq 0,
$$
which implies
\beqn \label{eq:D_t}
\forall \, t \geq 0, \quad \|D_t g_t \|_{L^1(\langle v \rangle m)} + \| g_t \|_{L^1(\langle v \rangle m)} \leq \eta^{-1} e^{-\alpha_0''t} \|h_0\|_{W^{s,1}_x W^{1,1}_v(\langle v \rangle m)}.
\eeqn
Then, we write 
$$
\begin{aligned}
t \, \nabla_x T_1(t) (\partial^\beta_x h_0) &= \int_{\R^3} k_\delta(v,v_*) \left[ (D_t g_t) - \nabla_{v_*} g_t \right] (x,v_*) \, dv_* \\
&= \AA_\delta(D_t g_t) + \AA^2_\delta g_t,
\end{aligned}
$$
Using~(\ref{eq:D_t}), we hence get
$$
\begin{aligned}
t \, \|\nabla_x T_1(t) (\partial^\beta_x h_0) \|_{L^1(K)}
&\leq C \, \left( \|D_t g_t \|_{L^1(\langle v \rangle m)} + \|g_t\|_{L^1(\langle v \rangle m)} \right) \\
&\leq C \, \eta^{-1} e^{-\alpha_0''t} \|h_0\|_{W^{s,1}_x W^{1,1}_v (\langle v \rangle m)}.
\end{aligned}
$$
Together with estimate~(\ref{eq:T_1'}) and Lemma~\ref{lem:Adelta}, for $s \geq 0$, we conclude that
$$
\|T_1(t) (\partial^\beta_x h_0)\|_{W^{1,1}_x W^{s+1,1}_v(K)} 
\leq \frac{C e^{-\alpha_0'' t}}{t} \| h_0 \|_{W^{s,1}_x W^{1,1}_v(\langle v \rangle m)},
$$
which in turn implies~(\ref{eq:T_1}). 

Now interpolating the last inequality and~(\ref{eq:T_1'}), for $s \geq 0$, we have
\beqn \label{eq:T_1''}
\|T_1(t) h_0\|_{W^{s+1/2,1}_{x,v}(K)} \leq \frac{C e^{-\alpha_0''t}}{\sqrt{t}} \| h_0 \|_{W^{s,1}_x W^{1,1}_v ( \langle v \rangle m)}.
\eeqn 
Putting together~(\ref{eq:T_1}) and~(\ref{eq:T_1''}), for $s \geq 0$, we obtain
\begin{align*}
\|T_2(t) h_0 \|_{W^{s+1/2,1}_{x,v} (K)} 
&\leq \int_0^t \|T_1(t-s) T_1(s) h_0 \|_{W^{s+1/2,1}_{x,v} (K)}  \, ds \\
&\leq C \int_0^t \frac{e^{-\alpha_0''(t-s)}}{(t-s)^{1/2}} \|T_1(s) h_0 \|_{W^{s,1}_x W^{1,1}_v(\langle v \rangle m)} \, ds \\
&\leq C \left( \int_0^t \frac{e^{-\alpha_0''(t-s)}}{(t-s)^{1/2}} e^{-\alpha_0 s} \, ds \right) \| h_0 \|_{W^{s,1}_{x,v}(\langle v \rangle m)} \\
&\leq C \sqrt{t} \,e^{-\alpha_0'' t} \| h_0 \|_{W^{s,1}_{x,v}( \langle v \rangle m)},
\end{align*}
which concludes the proof. 
\end{proof}

Let us now recall~\cite[Lemma~2.17]{GMM} which yields an estimate on the norms \linebreak $\|T_n\|_{\BBB(E_j,E_{j+1})}$ for $j=-1,0$.

\begin{lem} \label{lem:Tngeneral} 
Let $X$, $Y$ be two Banach space with $X \subset Y$ dense with continuous embedding, and consider $L \in \BBB(X)$, $\LL \in \BBB(Y)$ such that $\LL_{|X}=L$ and $a \in \R$. We assume that there exist some intermediate spaces 
$$
X = \EE_J \subset \EE_{J-1} \subset ... \subset \EE_2 \subset \EE_1 = Y, \quad J \geq 2
$$
such that, denoting $\AA_j := \AA_{|\EE_j}$ and $\BB_j := \BB_{|\EE_j}$
\begin{itemize}
\item[{\bf (i)}]  $(\BB_j - a)$ is hypodissipative and $\AA_j$ is bounded on $\EE_j$ for $1 \leq j \leq J$;
\item[{\bf (ii)}] there are some constants $\ell_0 \in \N^*$, $C \geq 1$, $K \in \R$, $\gamma \in [0,1)$ such that
$$
\forall \, t \geq 0, \quad \|T_{\ell_0}(t) \|_{\BBB(\EE_j, \EE_{j+1})} \leq C \frac{e^{Kt}}{t^\gamma},
$$
for $1 \leq j \leq J-1$, with the notation $T_\ell := \left(\AA S_\BB \right)^{(*\ell)}$.
\end{itemize}
Then for any $a'>a$, there exist some constructive constants $n \in \N$, $C_{a'}\geq 1$ such that
$$
\forall \,  t \geq 0, \quad \|T_n(t) \|_{\BBB(Y,X)} \leq C_{a'} e^{a't}.
$$
\end{lem}

Combining Lemmas~\ref{lem:dissip} and~\ref{lem:Tn}, we can apply Lemma~\ref{lem:Tngeneral} and deduce the following result:
\begin{lem} \label{lem:Tnfinal}
Let $\lambda$ be in $(0,\lambda_1]$. For any $\alpha_0' \in (0, \alpha_0)$, there exist some constructive constants $n \in \N$ and $C_{\alpha_0'} \geq 1$ such that 
$$
\forall \, t \geq 0, \quad \|T_n(t)\|_{\BBB(E_j,E_{j+1})} \leq C_{\alpha_0'} e^{-\alpha_0' t}, \quad j=-1,0.
$$
\end{lem}

\medskip

\subsection{Estimate on $\LL_\lambda - \LL_0$} 
Using estimates from the proof of Lemma~\ref{lem:dissip}, we can prove the following result:
\begin{lem} \label{lem:Llambda-L0}
There exists a function $\eta_1(\lambda)$ such that $\eta_1(\lambda) \xrightarrow[\lambda \rightarrow 0]{} 0$ and the difference $\LL_\lambda - \LL_0$ satisfies 
$$
\|\LL_\lambda - \LL_0\|_{\BBB(E_j,E_{j-1})} \le \eta_1(\lambda), \,\,\, j=0,1.
$$
\end{lem}

\begin{proof}
We have 
$$
\LL_\lambda - \LL_0 = \lambda^\gamma \Delta_v + \widehat{Q}_{e_\lambda} - \widehat{Q}_1.
$$

First, we have the following inequality:
\beqn \label{eq:lapla}
\|�\lambda^\gamma \Delta_v(h)\|_{E_{j-1}} \leq \lambda^\gamma \| h \|_{E_j}, \quad j=0,1.
\eeqn
Concerning the term $\widehat{Q} _{e_\lambda} - \widehat{Q}_1$, we have obtained in the proof of Lemma~\ref{lem:dissip} 
$$
\|(\widehat{Q}_{e_\lambda} - \widehat{Q}_1) h \|_{L^1_v(\langle v \rangle m)} \leq C \, \eps(\lambda) \|h\|_{L^1_v(\langle v \rangle^2 m)}, \quad \eps(\lambda) \xrightarrow[\lambda \rightarrow 0]{} 0.
$$
Again arguing as in the proof of Lemma~\ref{lem:dissip}, we obtain 
$$
\|\partial_v (\widehat{Q}_{e_\lambda} - \widehat{Q}_1) h \|_{L^1_v(\langle v \rangle m)} \leq C \, \eps(\lambda) \|h\|_{W^{1,1}_v(\langle v \rangle^2 m)}.
$$
We obtain the higher-order derivatives in the same way and we can conclude that 
\beqn \label{eq:Qelambda}
\|(\widehat{Q}_{e_\lambda} - \widehat{Q}_1) h \|_{E_0} \leq C \, \eps(\lambda) \|h\|_{E_1}.
\eeqn
Gathering~(\ref{eq:lapla}) and~(\ref{eq:Qelambda}), we deduce that 
$$
\|(\LL_\lambda - \LL_0) h\|_{E_0} \leq \eta_1(\lambda) \|h\|_{E_1}.
$$

Using the same method, we obtain: 
$$
\|(\LL_\lambda - \LL_0) h\|_{E_{-1}} \leq \eta_1(\lambda) \|h\|_{E_0}.
$$
\end{proof}

\smallskip
In the remaining part of the paper, $\delta$ is fixed (given by Lemma~\ref{lem:dissip}), we hence denote $\AA=\AA_\delta$ and $\BB_\lambda=\BB_{\lambda,\delta}$.

\medskip

\subsection{Semigroup spectral analysis of the linearized operator} 
In this section we shall state some results on the geometry of the spectrum of the linearized diffusive inelastic collision operator for a small diffusion parameter as well as a stability estimate for the semigroup associated to $\LL_\lambda$ in various weighted Sobolev spaces.

\begin{prop} \label{prop:linear}
There exists $\lambda_0 \in [0,1)$ such that for any $\lambda \in (0,\lambda_0]$, $\LL_\lambda$ satisfies the following properties in $E =W_x^{s,1} W^{2,1}_v (\langle v \rangle m)$, $s \in \N^*$:
\begin{itemize}
\item[{\bf (i)}] There exists $\mu_\lambda \in \R$ such that $ \Sigma(\LL_\lambda) \cap \Delta_{-\alpha_1}  = \{ \mu_\lambda, 0 \} $ where $\alpha_1$ is given by Theorem~\ref{theo:elastic}. Moreover, $0$ is a four-dimensional eigenvalue and $\mu_\lambda$ is a one-dimensional eigenvalue.
\item[{\bf (ii)}] $\mu_\lambda$ satisfies the following estimate 
\beqn \label{eq:mulambda}
\mu_\lambda = -C \lambda^\gamma + o(\lambda^\gamma)
\eeqn 
for some $C>0$.
\item[{\bf (iii)}] For any $\alpha \in (0, \hbox{\rm min} (\alpha_0, \alpha_1)) \setminus \{-\mu_\lambda \}$ (where $\alpha_0$ is provided by Lemma~\ref{lem:dissip}), the semigroup generated by $\LL_\lambda$ has the following decay property 
\beqn \label{eq:SGLlambda}
\forall \, t \geq 0, \quad \|S_{\LL_\lambda}(t)(I -  \Pi_{\LL_\lambda,0} -  \Pi_{\LL_\lambda,\mu_\lambda}) \|_{\BBB(E)} \leq C e^{-\alpha t}
\eeqn
for some $C>0$. 
\end{itemize}
\end{prop}

\noindent {\it Proof of Proposition~\ref{prop:linear}.}
\subsubsection*{Step 1 of the proof: the linearized elastic operator}
We recall hypodissipativity results for the semigroup associated to the linearized elastic Boltzmann equation which are straightforward adaptations of~\cite[Theorem~4.2]{GMM}. \begin{theo} \label{theo:elastic}
There are constructive constants $C \ge 1$, $\alpha_1>0$, such that the operator $\LL_0$ satisfies in $E_0$ and $E_1$:
$$
\Sigma(\LL_0) \cap \Delta_{-\alpha_1} = \{0\} \quad \text{and} \quad N(\LL_0) = Span \{ G_0, v_1G_0, v_2G_0, v_3G_0, |v|^2 G_0\}.
$$
Moreover, $\LL_0$ is the generator of a strongly continuous semigroup $h(t) = S_{\LL_0}(t) h_{in}$ in $E_0$ and $E_1$, solution to the initial value problem (\ref{eq:Bol2}) with $\lambda=0$, which satisfies:
$$
\forall \, t \ge 0, \quad \| h(t) - \Pi_{\LL_0, 0} h_{in} \|_{E_i} \le C e^{-\alpha_1 t }  \| h_{in} - \Pi_{\LL_0, 0} h_{in} \|_{E_i}, \quad i=0,1.
$$
\end{theo}

\smallskip
\subsubsection*{Step 2 of the proof: localization of spectrum of $\LL_\lambda$}

\begin{lem} 
Let us define $K_\lambda(z)$ for any $z \in \Omega := \Delta_{-\alpha_1} \setminus  \{0\} $ (where $\alpha_1$ is given by Theorem~\ref{theo:elastic}) by 
$$
K_\lambda(z) = (-1)^n \, (\LL_\lambda-\LL_0) \, \RR_{\LL_0}(z) \, \, (\AA \, \RR_{\BB_\lambda}(z))^n.
$$
 Then, there exists $\eta_2(\lambda)$ with $\eta_2(\lambda) \xrightarrow[\lambda \rightarrow 0]{} 0$ such that 
$$
  \forall \, z \in \Omega_\lambda:=\Delta_{-\alpha_1} \setminus \bar B(0, \eta_2(\lambda)), \quad \|K_\lambda(z)\|_{\BBB(E_0)} \leq \eta_2(\lambda).
$$
Moreover, there exists $\lambda_0 \in (0,\lambda_1]$ (where $\lambda_1$ is given by Lemma~\ref{lem:dissip}) such that  for any $\lambda \in [0,\lambda_0]$, we have
\begin{itemize}
\item[{\bf (i)}] $I+K_\lambda(z)$ is invertible for any $z \in \Omega_\lambda$
\item[{\bf (ii)}]  $\LL_\lambda - z$ is also invertible for any $z \in \Omega_\lambda$ and 
$$
\forall \, z \in \Omega_\lambda, \quad \RR_{\LL_\lambda}(z) = \UU_\lambda(z) \, (I+K_\lambda(z))^{-1} 
$$
where
$$
\UU_\lambda (z)  = \RR_{\BB_\lambda } (z) + ... +  (-1)^{n-1}  \, \RR_{\BB_\lambda} (z)\,  (\AA \, \RR_{\BB_\lambda}(z))^{n-1}  + (-1)^n  \, \RR_{\LL_0} \, (\AA \, \RR_{\BB_\lambda}(z))^n.
$$  
\end{itemize}
We thus deduce that 
$$
\Sigma(\LL_\lambda) \cap \Delta_{-\alpha_1} \subset B(0, \eta_2(\lambda)).
$$
\end{lem}

\begin{proof}
\medskip\noindent
{\it Step~1.}
We first notice that $(\AA \RR_{\BB_\lambda}(z))^n \in \BBB(E_0,E_1)$, $\RR_{\LL_{0}}(z) \in \BBB(E_1)$ and \linebreak $\LL_{\lambda} - \LL_{0} \in \BBB(E_1,E_0)$ for any $z \in \Omega$ because of Lemma~\ref{lem:Tnfinal}, Theorem~\ref{theo:elastic} and Lemma~\ref{lem:Llambda-L0}. Moreover, there exist $n \in \N$ and $C_0 >0$ such that $\|\RR_{\LL_{0}} (z) \|_{\BBB(E_1)} \le  C_0/|z|^{n}$ for any~$z$~in~$\Omega$. Indeed, we know from~\cite[paragraph I.5.3]{Kato} that in $E_1$, the following Laurent series
$$
\RR_{\LL_{0}} (z) = \sum_{k=-n}^{+\infty} \, z^k \, \CC_k
$$
where $\CC_k$ are some bounded operators in $\BBB(E_1)$, converges for $z$ close to $0$. We thus deduce the previous estimate on $\|\RR_{\LL_{0}} (z) \|_{\BBB(E_1)}$. Let us finally define $\eta_2(\lambda):= \left( C_0 C_{\lambda_1'} \eta_1(\lambda) \right)^{\frac{1}{n+1}}$ where $\lambda_1'$ is fixed in $(0, \lambda_1)$ and $C_{\lambda_1'}$ is given by Lemma~\ref{lem:Tnfinal}. We deduce that 
$$
\forall \, z \in \Omega_\lambda, \quad \|K_\lambda(z) \|_{\BBB(E_0)} \le \eta_1(\lambda) \frac{C_0}{\eta_2(\lambda)^{n}} C_{\lambda_1'} = \eta_2(\lambda).
$$
We then choose $\lambda_0 \in (0,\lambda_1]$ such that for any $\lambda \in (0,\lambda_0]$, $\eta_2(\lambda) < 1$. We hence obtain that $I+K_\lambda(z)$ is an invertible operator for any $\lambda \in (0,\lambda_0]$. Let us now consider $\lambda \in (0,\lambda_0]$. 

\medskip\noindent
{\it Step~2. $\UU_\lambda (z) \,  (I + K_\lambda(z))^{-1}$ is a right-inverse of $\LL_\lambda-z$ on $\Omega_\lambda$. } 
For any $z \in \Omega_\lambda$, we compute 
\bean
(\LL_\lambda - z) \, \UU_\lambda(z) 
&=&(\BB_\lambda - z + \AA) \, \{  \RR_{\BB_\lambda }(z)  + ... +  (-1)^{n-1}  \, \RR_{\BB_\lambda}(z) \,  (\AA \, \RR_{\BB_\lambda})^{n-1}(z)  \}
\\
&+&\,\, (-1)^n  \,   (\LL_\lambda - \LL_0 + \LL_0 - z) \,  \RR_{\LL_0}(z) \, (\AA \, \RR_{\BB_\lambda})^n (z) 
\\
&=& Id   +  K_\lambda(z).
\eean
Because of the previous step, we deduce that for $z \in \Omega_\lambda$, $\UU_\lambda (z) \,  (I + K_\lambda(z))^{-1}$ is a right-inverse of $\LL_\lambda-z$.

\medskip\noindent
{\it Step~3. There exists $z_0 \in \Omega_\lambda$ such that $\LL_\lambda-z_0$ is invertible on $\Omega_\lambda$. } 
Indeed, we write 
$$ 
\LL_\lambda-z_0 = (\AA \, \RR_{\BB_\lambda}(z_0) + I) \, (\BB_\lambda - z_0)
$$
where $(\AA \, \RR_{\BB_\lambda}(z_0) + I)$ is invertible for $\Re e \, z_0$ large enough because of Lemma~\ref{lem:dissip}. As a consequence, $\LL_\lambda-z_0$ is the product of two invertible operators, we hence obtain that  $\LL_\lambda-z_0$ is invertible. 

\medskip\noindent
{\it Step~4. $\LL_\lambda-z$ is invertible close to $z_0$.}
Since $\LL_\lambda-z_0$ is invertible on $\Omega_\lambda$, we have $\RR_{\LL_\lambda} (z_0) = \UU_\lambda (z_0) \,  (I + K_\lambda(z_0))^{-1}$. Moreover, if $\| \RR_{\LL_\lambda} (z_0) \| \le C$ for some $C>0$, then $\LL_\lambda - z$ in invertible on the disc $B(z_0, 1/C)$ with 
\beqn \label{eq:devResolvante}
\forall \, z \in B(z_0, 1/C), \quad \RR_{\LL_\lambda} (z) =  \RR_{\LL_\lambda} (z_0) \, \sum_{n=0}^{+\infty} \, (z-z_0)^n \, \RR_{\LL_\lambda} (z_0)^n ,
\eeqn
and arguing as before, $\RR_{\LL_\lambda} (z) = \UU_\lambda (z) \,  (I + K_\lambda(z))^{-1}$ on  $B(z_0, 1/C)$ since \linebreak $\UU_\lambda (z) \,  (I + K_\lambda(z))^{-1}$ is a right inverse of $\LL_\lambda - z$ for any $z \in \Omega_\lambda$. 

\medskip\noindent
{\it Step~5. $\LL_\lambda-z$ is invertible on $\Omega_\lambda$.}
For a given $z_1 \in \Omega_\lambda$, we consider a continuous path $\Gamma$ from $z_0$ to $z_1$ included in $\Omega_\lambda$, i.e. a continuous function $\Gamma : [0,1] \to \Omega_\lambda$ such that $\Gamma(0) = z_0$, $\Gamma(1) = z_1$.
We know that $(\AA \RR_{\BB_{\lambda}}(z))^\ell$, $1 \le \ell \le n-1$, ${\RR_{\LL_{0}} (z) } (\AA \RR_{\BB_{\lambda} }(z))^n$ and $(I+K_\lambda(z))^{-1}$ are locally uniformly bounded in $\mathscr{B}(E_0)$ on $\Omega_\lambda$, which implies
$$
\sup_{z \in \Gamma([0,1])} \| \UU_\lambda(z) (I+K_\lambda(z))^{-1} \|_{ \mathscr{B}(E_0)} := K <
\infty.
$$
Since $(\LL_\lambda - z_0)$ is invertible we deduce that $(\LL_\lambda - z)$ is invertible with $\RR_{\LL_\lambda}(z)$ locally bounded around $z_0$ with a bound $K$ which is uniform along $\Gamma$ (and a similar series expansion as in \eqref{eq:devResolvante}). By a continuation argument we hence obtain that $(\LL_\lambda-z)$ is invertible in $E_0$ all along the path $\Gamma$ with
\[
\RR_{\LL_\lambda}(z) = \UU_\lambda(z) (I+K_\lambda(z))^{-1} \ \mbox{ and } \ \| \RR_{\LL_\lambda}(z) \|_{\mathscr{B}(E_0)} \le K.
\]
Hence we conclude that $(\LL_\lambda-z_1)$ is invertible with $\RR_{\LL_\lambda}(z_1) =\UU_\lambda(z_1) (I+K_\lambda(z_1))^{-1}$.
\end{proof}
 
\smallskip
\subsubsection*{Step 3 of the proof: dimension of eigenspaces}
\begin{lem} \label{lem:projector} 
There exist a constant $C>0$ and a function $\eta_3(\lambda)$ such that
\beqn\label{eq:lem3-1}
\|\Pi_{\LL_\lambda,-\alpha_1} \|_{\BBB(E_0,E_1)} \le C,
\eeqn
and 
\beqn\label{eq:lem3-2}
\| \Pi_{\LL_\lambda,-\alpha_1}  -\Pi_{\LL_0,-\alpha_1}  \|_{\BBB(E_0)} \le \eta_3(\lambda), \quad \eta_3(\lambda) \xrightarrow[\lambda \rightarrow 0]{} 0.
\eeqn
It implies that for $\lambda$ close enough to $0$, we have 
$$
\hbox{\rm dim} \, \mbox{\rm R}(\Pi_{\LL_\lambda,-\alpha_1} ) = \hbox{\rm dim} \, \mbox{\rm R}(\Pi_{\LL_0,-\alpha_1} )=5.
$$
\end{lem}

The following lemma from~\cite[paragraph I.4.6]{Kato} is going to be useful for the proof.

\begin{lem} \label{lem:kato}
Let $X$ be a Banach space and $P$, $Q$ be two projectors in $\mathscr{B}(X)$ such that $\|P-Q \|_{\mathscr{B}(X)} < 1$. Then the ranges of $P$ and $Q$ are isomorphic. In particular, \linebreak $\hbox{\rm dim} (\hbox{\rm R}(P))= \hbox{\rm dim} (\hbox{\rm R}(Q))$.
\end{lem}


Let us now prove Lemma~\ref{lem:projector}. 
\begin{proof} Let $\Gamma := \{ z \in \C, \, |z| = \eta_2(\lambda)\}$ which is included in $\Omega$ for $\lambda$ small enough. We set $N:=2n$ and we define
\bean
\UU^0_\lambda :=\RR_{\BB_\lambda }  + ... +  (-1)^{N-1}  \, \RR_{\BB_\lambda} \,  (\AA \, \RR_{\BB_\lambda})^{N-1} \text{ and }
\quad \UU^1_\lambda :=  (-1)^N  \, \RR_{\LL_0} \, (\AA \, \RR_{\BB_\lambda})^N  , 
\eean
Notice that Lemma~\ref{lem:dissip} implies that $z \mapsto \RR_{\BB_\lambda }(z)$ is holomorphic in $\bar B (0, \eta_2(\lambda))$ for $\lambda$ small enough and consequently that $\int_\Gamma \UU^0_\lambda(z) \, dz =0$. We can then compute:
\bean
\Pi_{\LL_\lambda,-\alpha_1}  
&=& { i \over 2 \pi} \int_\Gamma \RR_{\LL_\lambda} (z) \, dz
\\
&=& { i \over 2 \pi} \int_\Gamma \UU_\lambda (z) \,   (I+K_\lambda(z))^{-1} \, dz
\\
&=& { i \over 2 \pi} \int_\Gamma \UU^0_\lambda(z) \, \{ I - K_\lambda (z)\, (I+K_\lambda(z))^{-1} \}\, dz
\\
&&  + { i \over 2 \pi} \int_\Gamma   \UU^1_\lambda(z)  \, (I+K_\lambda(z))^{-1} \, dz
\\
&=& { 1 \over 2i \pi} \int_\Gamma \UU^0_\lambda(z) \,   K_\lambda (z)\, (I+K_\lambda(z))^{-1} \, dz
\\
&&  + (-1)^n \, { i \over 2 \pi} \int_\Gamma   \RR_{\LL_0}(z) \, (\AA \, \RR_{\BB_\lambda}(z))^N  \, (I+K_\lambda(z))^{-1} \, dz.
\eean
Since $(\AA \RR_{\BB_\lambda}(z))^N$ appears in the two parts of the expression of $\Pi_{\LL_\lambda,-\alpha_1} $,  we deduce that~\eqref{eq:lem3-1} holds. \\

Concerning the estimate on $\Pi_{\LL_0,-\alpha_1}  - \Pi_{\LL_\lambda,-\alpha_1} $, we begin by writing
$$
 \RR_{\LL_0} (z) =  \RR_{\BB_0} (z)+ ... +  (-1)^{N-1}  \, \RR_{\BB_0}(z) \,  (\AA \, \RR_{\BB_0}(z))^{N-1}  + (-1)^N  \, \RR_{\LL_0}(z) \, (\AA \, \RR_{\BB_0}(z))^N
 $$
 which implies that
\bean
\Pi_{\LL_0,-\alpha_1} 
&=& { i \over 2 \pi} \int_\Gamma \RR_{\LL_0} (z) \, dz
\\
&=&  (-1)^n  \,  { i \over 2 \pi} \int_\Gamma  \RR_{\LL_0}(z) \, (\AA \, \RR_{\BB_0}(z))^N   \, dz.
\eean
Finally, we deduce that
\bean
&& \Pi_{\LL_0,-\alpha_1}  - \Pi_{\LL_\lambda,-\alpha_1}   \\
&=& (-1)^n  \,  { i \over 2 \pi} \int_\Gamma  \RR_{\LL_0}(z) \, \{(\AA \, \RR_{\BB_0}(z))^N  
-  (\AA \, \RR_{\BB_\lambda}(z))^N  \, (I+K_\lambda(z))^{-1} \}\, dz
\\
&& -  { 1 \over 2 i\pi} \int_\Gamma  \UU^0_\lambda(z) \,  K_\lambda (z)\, (I+K_\lambda(z))^{-1} \, dz.
\eean

Since $K_\lambda(z)$ appears in the second term, we deduce that it is bounded by $\eta_2(\lambda)$. Concerning the first term, we rewrite it as
\bean
(\AA \, \RR_{\BB_0}(z))^{2n} - (\AA \, \RR_{\BB_\lambda}(z))^{2n}  + (\AA \, \RR_{\BB_\lambda}(z))^{2n}  ( I -  (I+K_\lambda(z))^{-1} ).
\eean
The second part of this expression is bounded by $\eta_2(\lambda)/(1-\eta_2(\lambda))$ because of the bound on the norm of $K_\lambda$. The first part can be written as 
\bean
\sum_{k=0}^{2n} \, (\AA \, \RR_{\BB_0}(z))^{k} \, \AA \, (\RR_{\BB_0}(z) - \RR_{\BB_\lambda}(z)) \, (\AA \, \RR_{\BB_\lambda}(z))^{2n-k-1}.
\eean
In addition, the bound on the norm of $\BB_\lambda - \BB_0$ given by Lemma~\ref{lem:Llambda-L0} gives a bound on the norm of $\RR_{\BB_\lambda}(z) -\RR_{\BB_0}(z)$ because
$$
\RR_{\BB_1}(z)- \RR_{\BB_\lambda}(z) = \RR_{\BB_\lambda}(z) \, (\BB_\lambda - \BB_0) \, \RR_{\BB_0}(z).
$$
Since for all $k$, $0 \leq k \leq 2n$ we have $k \geq n$ or $2n - k -1 \geq n$, we can use Lemma~\ref{lem:Tnfinal} and conclude that $(\AA \, \RR_{\BB_0}(z))^{2n} - (\AA \, \RR_{\BB_\lambda}(z))^{2n} $ is bounded by $C \eta_1(\lambda)$, which concludes the proof of \eqref{eq:lem3-2}.

The last part of Lemma~\ref{lem:projector} is nothing but Lemma~\ref{lem:kato} because for $\lambda$ close enough to~$0$, $ \eta_3(\lambda) < 1$. 
\end{proof}

We can now finish the proof of Theorem~\ref{theo:linearmainresult}-{\bf (i)}. The previous lemma implies that there exist $\xi_1, ..., \xi_5 \in \C$ such that
$$
\Sigma(\LL_\lambda) \cap \Delta_{-\alpha_1} = \{\xi_1, ... , \xi_5 \}. 
$$
Moreover, we know that $0$  is a four-dimensional eigenvalue due to the conservation of mass and momentum. Since the operator is real, we can deduce that there exists $\mu_\lambda \in \R$ such that
$$
\Sigma(\LL_\lambda) \cap \Delta_{-\alpha_1} = \{0, \mu_\lambda \}.
$$

\smallskip
\subsubsection*{Step 4 of the proof: fine study of spectrum close to 0}
Concerning the case of a constant coefficient of inelasticity, we refer to~\cite[Section~5.2, Step~2]{MM2} for the proof of Theorem~\ref{theo:linearmainresult}-{\bf(ii)} (the first order expansion of $\mu_\lambda$~(\ref{eq:mulambda})). 
Let us deal with the non-constant case. 

We first denote $\phi_0$ the energy eigenfunction of the the elastic linearized operator associated to $0$ such that $\|\phi_0\|_{L^1_v(\langle v \rangle ^2)} =1$. We also denote $\Pi_0$ the projection on $\R \, \phi_0$ and $\pi_0 \psi$ the coordinate of $\Pi_0 \, \psi$ on $\R \,  \phi_0$ i.e. $\Pi_0 \, \psi = (\pi_0 \psi) \,  \phi_0$. Finally, we denote $\phi_\lambda$ the unique eigenfunction associated to $\mu_\lambda$ such that $\|\phi_\lambda\|_{L^1_v(\langle v \rangle^2)}=1$ and $\pi_0 \phi_\lambda \geq 0$. 

By integrating in $v$ the eigenvalue equation related to $\mu_\lambda$
$$
\LL_\lambda \, \phi_\lambda = \mu_\lambda \, \phi_\lambda
$$
against $|v|^2$, we get
\beqn \label{eq:energy}
2 \int_{\R^3} \widetilde{Q}_{e_\lambda}(G_\lambda, \phi_\lambda) \, |v|^2 \, dv + \lambda^\gamma \int_{\R^3} \Delta_v\,  \phi_\lambda \, |v|^2 \, dv = \mu_\lambda \, \EE(\phi_\lambda).
\eeqn

We now compute the left-hand side of~(\ref{eq:energy}). By a classical computation which uses~(\ref{eq:lossenergy}), we have:
$$
2 \int_{\R^3} \widetilde{Q}_{e_\lambda}(G_\lambda, \phi_\lambda) \, |v|^2 \, dv = - \int_{\R^3 \times \R^3\times \Sp^2} |u|^3 G_{\lambda_*} \phi_\lambda \, \frac{1-\widehat{u} \cdot \sigma}{4} (1-{e_\lambda}^2) \, d\sigma \, dv_* \, dv
$$
and using polar coordinates
$$
\int_{\Sp^2}   { \frac{1-\widehat{u} \cdot \sigma}{4} \left(1-{e_\lambda}^2 \left( |u| \sqrt{\frac{1-\widehat{u} \cdot \sigma}{2}} \right) \right) } \, d\sigma = 4 \pi \int_0^1 \left( 1 - {e_\lambda}^2 \left( |u| y \right) \right) y^3 \, dy.
$$
Let us define
$$
\psi_e(r) := 4 \pi \, r^{3/2} \,  \int_0^1 (1-e^2(\sqrt{r} z)) \, z^3 \, dz,
$$
we can compute $\psi_{e_\lambda} (r) = \lambda^{-3} \, \psi_e(\lambda^2 r)$. We deduce that 
$$
2 \int_{\R^3} \widetilde{Q}_{e_\lambda}(G_\lambda, \phi_\lambda) \, |v|^2 \, dv = - \frac{1}{ \lambda^{3}} \, \int_{\R^3 \times R^3} G_{\lambda_*} \phi_\lambda \, \psi_e(\lambda^2 |u|^2) \, dv_* \, dv.
$$
We also have 
$$
\int_{\R^3} \Delta_v\,  \phi_\lambda \, |v|^2 \, dv = 6 \int_{\R^3} \phi_\lambda \, dv = 6 \, \rho(\phi_\lambda).
$$

Dividing~(\ref{eq:energy}) by $\lambda^\gamma$, we hence obtain
\beqn \label{eq:energy2}
- \frac{1}{ \lambda^{3+\gamma}} \, \int_{\R^3 \times \R^3} G_{\lambda_*} \phi_\lambda \, \psi_e(\lambda^2 |u|^2) \, dv_* \, dv  + 6 \, \rho(\phi_\lambda) = \frac{1}{\lambda^{\gamma}} \, \mu_\lambda \, \EE(\phi_\lambda).
\eeqn

\smallskip
We would like to make $\lambda$ tend to $0$ in~(\ref{eq:energy2}). To do that, we introduce the following notations:
$$
I_\lambda(f,g) := \int_{\R^3\times\R^3} f_* g \, \zeta_\lambda (|u|^2) \, dv_* \, dv \quad \text{with} \quad \zeta_\lambda (r^2) = \frac{1}{\lambda^{3+\gamma}} \, \psi_e (\lambda^2 r^2),
$$
and 
$$
I_0(f,g) := \int_{\R^3\times\R^3} f_* g \, \zeta_0 (|u|^2) \, dv_* \, dv \quad \text{with} \quad \zeta_0 (r^2) = \frac{a}{4+\gamma} \, r^{3+\gamma}.
$$
Let us now prove that $I_\lambda (G_\lambda, \phi_\lambda)$ tends to $I_0(G_0,\phi_0)$ as $\lambda$ tends to $0$. We state the following lemma which is going to be useful. We do not prove it here because the proof is the same as the one of~\cite[Lemma~5.17]{MM}.

\begin{lem} \label{lem:eigenfunctions}
Let $k$, $q \in \N$. We have the following result:
$$
\|�\phi_\lambda - \phi_0 \|_{W^{k,1}_v(\langle v \rangle ^q m)} \xrightarrow[\lambda \rightarrow 0]{} 0.
$$
\end{lem}

To prove that $I_\lambda (G_\lambda, \phi_\lambda)$ tends to $I_0(G_0,\phi_0)$ as $\lambda$ tends to $0$, let us write the following inequality:
\begin{align*}
|I_\lambda (G_\lambda, \phi_\lambda) - I_0(G_0, \phi_0)| &\leq |I_\lambda (G_\lambda, \phi_\lambda) - I_0(G_\lambda, \phi_\lambda)| + |I_0 (G_\lambda, \phi_\lambda) - I_0(G_0, \phi_0)| \\
&=: J^1_\lambda + J^2_\lambda.
\end{align*}

We first deal with $J^2_\lambda$:
\begin{align*}
J^2_\lambda = &\left| \int_{\R^3\times\R^3} (G_{\lambda_*} \phi_\lambda - G_{0_*} \phi_0) \,  \zeta_0(|u|^2) \, dv_* \, dv \right| \\
\leq &\, \,  C \int_{\R^3\times\R^3}  | G_{\lambda_*} - G_{0_*} | \,\phi_0 \,  \langle v \rangle ^{3+\gamma} \, \langle v_* \rangle^{3+\gamma} \, dv_* \, dv \\
&+ C \int_{\R^3\times \R^3}  G_{\lambda_*} \, |\phi_\lambda - \phi_0 |�\, \langle v \rangle ^{3+\gamma} \, \langle v_* \rangle^{3+\gamma} \, dv_* \, dv \\
\leq & \, \, C \left( \|G_\lambda - G_0 \|_{L^1_v ( \langle v \rangle ^{3+\gamma})} \,   \|�\phi_0 \|_{L^1_v ( \langle v \rangle ^{3+\gamma})} +  \|G_\lambda \|_{L^1_v ( \langle v \rangle ^{3+\gamma})}  \, \|\phi_\lambda - \phi_0 \|_{L^1_v ( \langle v \rangle ^{3+\gamma})} \right) \\
\leq & \, \, C \left( \|G_\lambda - G_0 \|_{L^1_v ( \langle v \rangle ^{3+\gamma})} +  \|\phi_\lambda - \phi_0 \|_{L^1_v ( \langle v \rangle ^{3+\gamma})} \right) \xrightarrow[\lambda \rightarrow 0]{} 0
\end{align*}
because of Lemmas~\ref{lem:Glambda-G0} and~\ref{lem:eigenfunctions}. 

Let us now establish an estimate on $J^1_\lambda$:
$$
J^1_\lambda \leq \int_{\R^3\times\R^3}   G_{\lambda_*} \phi_\lambda \, |\zeta_\lambda (|u|^2) - \zeta_0(|u|^2) | \, dv_* \, dv =: D_\lambda.
$$
We can rewrite the difference $\zeta_\lambda(r^2) - \zeta_0(r^2)$ in the following way:
$$
\zeta_\lambda(r^2) - \zeta_0(r^2) = \frac{r^{3+\gamma}}{2} \int_0^1 \left( \frac{1- e^2(\lambda r z)}{(\lambda r z)^\gamma} - 2a \right) z^{3+\gamma}�\, dz,
$$
which allows us to get an estimate on this difference using Assumption~\ref{ass}-(\ref{ass:3}). There exists a constant $C >0$ such that 
$$
\forall \, \lambda \in (0,1], \quad \forall r >0,  \quad |\zeta_\lambda(r^2) - \zeta_0(r^2)| \leq  C \left( r^{3 + 2\gamma} \,\lambda^\gamma  + r^{3+ \gamma + \overline{\gamma}}\, \lambda^{\overline{\gamma}}  + r^{3+ \overline{\gamma}}\, \lambda^{\overline{\gamma} - \gamma}  \right).
$$
Denoting $\widetilde{\gamma} := \min (\gamma, \overline{\gamma} - \gamma)$, we can deduce that 
\begin{align*}
D_\lambda &\leq \int_{\R^3\times \R^3} G_{\lambda_*} \phi_\lambda \, \lambda^{\widetilde{\gamma}} |u|^{3+ \gamma + \overline{\gamma}}�\, dv_* \, dv \\
& \leq C \, \lambda^{\widetilde{\gamma}} \, \|G_\lambda\|_{L^1_v(\langle v \rangle ^{3+ \gamma + \overline{\gamma}})} \|\phi_\lambda\|_{L^1_v(\langle v \rangle ^{3+ \gamma + \overline{\gamma}})} \\
& \leq C \, \lambda^{\widetilde{\gamma}}.
\end{align*}
It yields the result: $J^1_\lambda\xrightarrow[\lambda \rightarrow 0]{} 0$.

\smallskip
We can now make $\lambda$ tend to $0$ in~(\ref{eq:energy2}). Using the previous result \linebreak $I_\lambda (G_\lambda, \phi_\lambda) \rightarrow I_0(G_0, \phi_0)$, the fact that the mass of $\phi_0$ is $0$ and the convergences $G_\lambda \rightarrow G_0$ and $\phi_\lambda \rightarrow \phi_0$ (Lemmas~\ref{lem:Glambda-G0} and~\ref{lem:eigenfunctions}), we deduce that 
$$
\frac{\mu_\lambda}{\lambda^\gamma} \EE(\phi_0) = - I_0(G_0, \phi_0) + o(1).
$$
We finally conclude that there exists a constant $C>0$ such that 
$$
\mu_\lambda = -C \lambda^\gamma + o(\lambda^\gamma).
$$

\smallskip
\subsubsection*{Step 5 of the proof: semigroup decay} 
In order to get our semigroup decay, we are going to apply the following quantitative spectral mapping theorem which comes from~\cite{MS}. We give here a simpler version and hence give the proof which is easier in this case.

\begin{prop} \label{prop:SGdecay}
Consider a Banach space $X$ and an operator $\Lambda \in \CCC(X)$ so that \linebreak $\Lambda = \AA + \BB$ where $\AA \in \BBB(X)$ and $\BB - a$ is hypodissipative on $X$ for some $a \in \R$. We assume furthermore that there exists a family $X_j$, $1 \leq j \leq m$, $m \geq 2$ of intermediate spaces such that 
$$
X_m \subset \hbox{\rm D}(\Lambda^2)  \subset X_{m-1} \subset ... \subset X_2 \subset X_1 = X, 
$$
and a family of operators $\Lambda_j, \AA_j, \BB_j \in \CCC(X_j)$ such that
$$
\Lambda_j = \AA_j + \BB_j, \quad \Lambda_j = \Lambda_{|X_j}, \quad  \AA_j = \AA_{|X_j}, \quad \BB_j = \BB_{|X_j},
$$
and that there holds
\begin{itemize}
\item[{\bf (i)}] $(\BB_j - a)$ is hypodissipative on $X_j$; 
\item[{\bf (ii)}] $\AA_j \in \BBB(X_j)$;
\item[{\bf (iii)}] there exists $n \in \N$ such that $T_n(t) := \left( \AA S_\BB(t) \right) ^{(\ast n)}$ satisfies $\|T_n(t)\|_{\BBB(X,X_m)} \leq C e^{at}$. 
\end{itemize}
Then the following localization of the principal part of the spectrum 
\begin{itemize}
\item[{\bf (1)}] there are some distinct complex numbers $\xi_1, ..., \xi_k \in \Delta_a$, $k \in \N$ (with the convention $\{ \xi_1, ... \xi_k \} = \emptyset$ if $k=0$) such that one has 
$$
\Sigma(\Lambda)  \cap \Delta_a = \{\xi_1, ..., \xi_k \}  \subset \Sigma_d(\Lambda).
$$
\end{itemize}
implies the following quantitative growth estimate on the semigroup 
\begin{itemize}
\item[{\bf (2)}] for any $a' \in (a,\infty) \setminus \{ \Re e \, \xi_j, j=1,...,k\}$, there exists some constructive constant $C_{a'} >0$ such that 
$$
\forall t \geq 0, \quad \left\| S_\Lambda (t) - \sum_{j=1}^k S_{\Lambda}(t) \Pi_{\Lambda, \xi_j} \right\|_{\BBB(X)} \leq C_{a'} e^{a't}.
$$
\end{itemize}
\end{prop}

\begin{proof}
We have the following representation formula (see for instance the proof of~\cite[~Theorem~2.13]{GMM}):
$$
S_\Lambda(t) f = \sum_{j=1}^k S_{\Lambda,\xi_j}(t) f + \sum_{\ell=0}^{n+1} (-1)^\ell S_\BB * (\AA S_\BB)^{(*\ell)} (t) f + \ZZ(t) f,
$$
for any $f \in \mbox{\rm D}(\Lambda)$ and $t \geq 0$, where 
$$
\ZZ(t)f := \lim_{M \rightarrow \infty} \frac{(-1)^n}{2i\pi} \int_{a'- i M}^{a' + iM} {e^{zt} \, \RR_\Lambda(z) \, (\AA \RR_\BB(z))^{n+2}  f \, dz}.
$$
On the one hand, we know from \textbf{(i)} and \textbf{(ii)} that 
$$
\forall \, \ell = 0, ..., n+1, \quad \| S_\BB * (\AA S_\BB)^{(*\ell)} (t) \|_{\BBB(X)} \leq C_{a'} e^{a't}.
$$
On the other hand, because of \textbf{(iii)}, we have
$$
\sup_{z \in a' + i\R} \| (\AA \RR_\BB)^n(z) \|_{\BBB(X, \textrm{\rm D}(\Lambda^2))} \leq K_{a'}^1
$$
and because of \textbf{(1)}, since $\Lambda$ generates a semigroup,
$$
\sup_{z \in a' + i\R} \| \RR_\Lambda (z) \|_{\BBB(X)} \leq K_{a'}^2.
$$
Then, we are going to use the resolvent identity
\beqn \label{eq:resolvid}
\forall \, z \notin \Sigma(\BB), \quad \RR_\BB(z) = z^{-1} [\RR_\BB(z) \BB - I ]
\eeqn
to get an estimate on $\|(\AA \RR_\BB)^2(z)\|_{\BBB( \textrm{\rm D}(\Lambda^2),X)}$ if $|z| \geq 1$. 
Using twice (\ref{eq:resolvid}), we obtain
$$
\forall \, z \in \C, \, |z|\geq 1, \quad \|(\AA \RR_\BB)^2(z) f \|_X \leq K_{a'}^3 |z|^{-2} \|f\|_{\textrm{\rm D}(\BB^2)}
$$
and we notice that $ \mbox{\rm D}(\BB^2) =  \mbox{\rm D}(\Lambda^2)$ because $\AA$ is bounded. We finally obtain
$$
\forall \, z \in \C, \, |z|\geq 1, \quad \|(\AA \RR_\BB)^2(z) f \|_X \leq 2K_{a'}^3 \frac{1}{1+|z|^2} \|f\|_{\textrm{\rm D}(\Lambda^2)}.
$$
Moreover, we also have
$$
\forall \, z \in \C, \, |z|\leq 1, \quad \|(\AA \RR_\BB)^2(z) f \|_X \leq K_{a'}^4\frac{1}{1+|z|^2} \|f\|_{ \textrm{\rm D}(\Lambda^2)}.
$$
All together, we deduce that
$$
\|\ZZ(t) \|_{\BBB(X)} \leq K_{a'} \frac{e^{a't}}{2\pi} \int_{\R} \frac{dy}{1+y^2},
$$
which yields the result. 
\end{proof}

We can now prove the estimate on the semigroup decay~(\ref{eq:SGLlambda}). We apply Proposition~\ref{prop:SGdecay} with $a:=\mbox{\rm max} (-\alpha_0, -\alpha_1) <0$. We have $E_1 \subset \hbox{\rm D}(\LL_\lambda^2) \subset E_0 \subset E_{-1}$. Assumptions \textbf{(i)}, \textbf{(ii)} and \textbf{(iii)} are nothing but Lemmas~\ref{lem:dissip},~\ref{lem:Adelta} and~\ref{lem:Tnfinal}. And \textbf{(1)} is given by the previous steps of the proof. We hence conclude that we have the decay result~(\ref{eq:SGLlambda}) for any $\alpha' \in (0, \hbox{\rm min} (\alpha_0, \alpha_1)) \setminus \{-\mu_\lambda \}$.

\begin{theo}  \label{theo:linearmainresult}
There exists $\lambda_0 \in [0,1)$ such that for any $\lambda \in (0,\lambda_0]$, $\LL_\lambda$ satisfies the following properties in $\EE=W_x^{s,1} L^1_v (m)$, $s \in \N^*$:
\begin{itemize}
\item[{\bf (i)}] The spectrum $ \Sigma(\LL_\lambda) $ satisfies the separation property: $ \Sigma(\LL_\lambda) \cap \Delta_{-\alpha_1}  = \{ \mu_\lambda, 0 \} $ where $\alpha_1$ is given by Theorem~\ref{theo:elastic}, $\mu_\lambda$ is given by Proposition \ref{prop:linear} and satisfies~\eqref{eq:mulambda}. 
\item[{\bf (ii)}] For any $\alpha \in (0, \hbox{\rm min} (\alpha_0, \alpha_1)) \setminus \{-\mu_\lambda \}$ (where $\alpha_0$ is provided by Lemma~\ref{lem:dissip}), the semigroup generated by $\LL_\lambda$ has the following decay property 
\beqn \label{eq:SGLlambda2}
\forall \, t \geq 0, \quad \|S_{\LL_\lambda}(t)(I -  \Pi_{\LL_\lambda,0} -  \Pi_{\LL_\lambda,\mu_\lambda}) \|_{\BBB(\EE)} \leq C e^{-\alpha t}
\eeqn
for some $C>0$. 
\end{itemize}
\end{theo}

\noindent {\it Proof of Theorem \ref{theo:linearmainresult}.} The proof relies on the combination of Proposition \ref{prop:linear} and the Theorem~2.13 of enlargement on the functional space of semigroup decay from \cite{GMM}. Our small space is $E$ and our large space is $\EE$. The assumptions of the theorem are clearly fulfilled thanks to Lemmas \ref{lem:dissip} and \ref{lem:Adelta} and thus yield the conclusion.
\begin{cor} \label{cor:linearmainresult}
There exists $\lambda_* \in (0,\lambda_0]$ such that for any $\lambda \in (0,\lambda_*]$ and $\alpha \in (0, - \mu_\lambda)$, we have
\beqn \label{eq:SGlambdaalpha}
\|S_{\LL_\lambda}(t) (I- \Pi_{\LL_\lambda,0})\|_{\BBB(\EE)} \leq C e^{-\alpha t}.
\eeqn
\end{cor}
\begin{proof}
The proof is immediate using the first order expansion of $\mu_\lambda$~(\ref{eq:mulambda}) which implies that $\mu_\lambda<0$ for $\lambda$ close enough to $0$. 
\end{proof}

\medskip
 \subsection{A dissipative Banach norm for the full linearized operator}
Let us define a new norm on $\EE$ by 
\beqn \label{eq:newnorm}
\Nt h \Nt_{\EE} := \eta \|h\|_{\EE} + \int_{0}^{+\infty} \|S_{\LL_\lambda}(\tau)(I- \Pi_{\LL_\lambda, 0}) h \|_{\EE} \, d\tau, \quad \eta>0,
\eeqn
which is well-defined thanks to estimate (\ref{eq:SGlambdaalpha}) for $\lambda$ small enough.

\begin{prop} \label{prop:dissipnorm}
Consider $\lambda \in (0,\lambda_*]$. There exist $\eta >0$ and $\alpha_2>0$ such that for any $h_{in} \in \EE$, $\Pi_{\LL_\lambda, 0} h_{in} =0$, the solution $h(t):=S_{\LL_\lambda}(t) h_{in}$ to the initial value problem~(\ref{eq:Bol2}) satisfies:
$$
\forall \, t \geq 0, \quad \frac{d}{dt} \Nt h_t \Nt_{\EE} \leq -\alpha_2  \Nt h_t \Nt_{\EE^1},
$$
where $\EE^1 := W^{s,1}_x L^1_v(\langle v \rangle m)$ and $\Nt \cdot \Nt_{\EE^1}$ is defined as in~(\ref{eq:newnorm}).  
\end{prop}


\begin{proof}
Notice that from the decay property of $\LL_\lambda$ provided by~(\ref{eq:SGlambdaalpha}), we have that the norms $\| \cdot \|_{\EE}$ and $\Nt \cdot  \Nt_{\EE}$ are equivalent for any $\eta >0$. 

Let us now compute the time derivative of the norm $\EE$ along $h_t$ where $h_t$ solves the linear evolution problem~(\ref{eq:Bol2}). Observe that $\Pi_{\LL_\lambda,0} h_t = 0$ due to the mass and momentum conservation properties of the linearized equation. Since the $x$-derivatives commute with the equation, we can set $s=0$. We thus only treat the case $L^1_x L^1_v (m)$. 
We compute
$$
\begin{aligned}
\frac{d}{dt} \Nt h_t \Nt_{L^1_x L^1_v (m)} &= \eta \int_{\R^3} \left( \int_{\T^3} \LL_\lambda(h_t) \, \text{sign} (h_t )\, dx \right) m \, dv + \int_0^{\infty} \frac{\partial}{\partial t} \| h_{t+\tau} \|_{L^1_x L^1_v ( m)} \, d\tau \\&=: I_1 + I_2.
\end{aligned}
$$
Concerning the first term, arguing as in the proof of Lemma~\ref{lem:dissip1}, we have from the dissipativity of $\BB_{\lambda}$ and the bounds on $\AA$
$$
I_1 \leq \eta \, \left(C \| h_t \|_{L^1_x L^1_v (m)} - K \| h_t \|_{L^1_x L^1_v ( \langle v \rangle m)}\right)
$$
for some constants $C, K >0$. 

The second term is computed exactly:
$$
I_2 = \int_0^{\infty} \frac{\partial}{\partial t} \|h_{t+\tau}\|_{L^1_x L^1_v ( m)} \, d\tau =  \int_0^{\infty} \frac{\partial}{\partial \tau} \|h_{t+\tau}\|_{L^1_x L^1_v ( m)} \, d\tau = -\|h_t\|_{L^1_x L^1_v (  m)}.
$$

The combination of the two last equations yields the desired result by choosing $\eta$ small enough. 
\end{proof}

\bigskip


\section{The nonlinear Boltzmann equation} 
\label{sec:nonlinear-equation}
\setcounter{equation}{0}
\setcounter{theo}{0}


In this section, the integer $s$ is fixed such that $s>6$ and we recall that 
$$
\EE = W^{s,1}_x L^1_v(m) \quad \text{and} \quad \EE^1 = W^{s,1}_x L^1_v(\langle v \rangle m).
$$
\subsection{The bilinear estimates}
We first establish bilinear estimates on the nonlinear term in equation~(\ref{eq:Bol1}). 
\begin{lem} \label{lem:bilinear}
In the space $X^q := W^{\sigma,1}_v W^{s,1}_x ( \langle v \rangle ^q m)$ with $s, \, \sigma \in \N$, $s > 6$ and $q \in \N$, the collision operator $Q$ satisfies 
$$
\|Q_{e_\lambda}(g,f)\|_{X^q} \leq C\left(\|g\|_{X^{q+1}} \|f\|_{X^{q}} +\|g\|_{X^{q}} \|f\|_{X^{q+1}}\right)
$$
for some constant $C>0$, where $X^{q+1}$ is defined as $X^q$. 
\end{lem}

\begin{proof}
Let us first consider the velocity aspect only of the norm with $\sigma = 0$. Concerning the case of a constant coefficient of inelasticity, we use that the elastic collision operator $Q_1$ satisfies (cf.~\cite{Mou})
$$
\|Q_1(g,f)\|_{L^1_v(m)} \leq C \left( \|f\|_{L^1_v(m)} \|g\|_{L^1_v(\langle v \rangle m)} +  \|f\|_{L^1_v(\langle v \rangle  m)} \|g\|_{L^1_v(m)}\right).
$$
First, it can be straightforwardly adapted to the case $L^1(\langle v \rangle^q m)$. Then, if $v'_\lambda$ and $v'_0$ denotes the post-collisional velocities in the inelastic case and in the elastic case with obvious notations, using the fact that we both have 
$$
|v'_\lambda|^2 \leq |v|^2 + |v_*|^2
$$
and
$$
|v'_0|^2 \leq |v|^2 + |v_*|^2,
$$
the same proof can be done in the inelastic case. We hence obtain that
\beqn \label{eq:bilinear}
\|Q_{e_\lambda}(g,f)\|_{L^1_v(\langle v \rangle ^q m)} \leq C \left( \|f\|_{L^1_v(\langle v \rangle ^q m)} \|g\|_{L^1_v(\langle v \rangle ^{q+1}m)} +  \|f\|_{L^1_v(\langle v \rangle ^{q+1} m)} \|g\|_{L^1_v(\langle v \rangle ^q m)}\right).
\eeqn
Then, from property~(\ref{derivQ}) and inequality~(\ref{eq:bilinear}), we deduce that
\bean
\|Q_{e_\lambda}(g,f)\|_{W^{\sigma,1}_v(\langle v \rangle ^q m)} &\leq& C \left( \|f\|_{W^{\sigma,1}_v(\langle v \rangle ^q m)} \|g\|_{W^{\sigma,1}_v(\langle v \rangle ^{q+1}m)} +  \right. \\
  && \left. \|f\|_{W^{\sigma,1}_v(\langle v \rangle ^{q+1} m)} \|g\|_{W^{\sigma,1}_v(\langle v \rangle ^q m)} \right)
\eean
as well as similar results from the other estimates.

As a final step, we consider the $x$ aspect of the norm. We use the Sobolev embedding $W^{s/2,1}_x(\T^3) \subset L^\infty_x(\T^3)$ with continuous embedding since $s > 6$ to conclude.
\end{proof}

\medskip
\subsection{The main results}
Let us now give some results on the Cauchy problem, stability and relaxation to equilibrium for the solutions to the full non-linear problem. We consider first the close-to-equilibrium regime (Theorem~\ref{theo:existence}), and then the weakly inhomogeneous regime (Theorem~\ref{theo:weak2}).  
\begin{theo} [Perturbative solutions close to equilibrium] \label{theo:existence}
Let us consider a restitution coefficient $e$ constant or satisfying Assumptions \ref{ass}-\ref{ass2} and $\lambda \in [0, \lambda_*]$ (where $\lambda_*$ is given by Corollary~\ref{cor:linearmainresult}). There is some constructive constant $\eps >0$ such that for any initial datum $f_{in} \in \EE$ satisfying
$$
\|f_{in} - G_\lambda \|_{\EE} \leq \eps,
$$
and $f_{in}$ has the same global mass and momentum as the equilibrium $G_\lambda$, there exists a unique global solution 
$f \in L^\infty_t(\EE) \cap L^1_t(\EE^1)$ to~(\ref{eq:Bol1}). 

This solution furthermore satisfies that for any ${\alpha} \in (0, -\mu_\lambda)$:
$$
\forall \, t \geq 0, \quad \| f_t - G_\lambda\|_{\EE} \leq C e^{-{\alpha} t} \, \|f_{in} - G_\lambda \|_{\EE}
$$
for some constructive constant $C \ge 1$. 
\end{theo}

\smallskip
For the following theorem, we only consider the case of a constant restitution coefficient, namely, $e_\lambda(\cdot)$ is constant equal to $1-\lambda$.
\begin{theo}[Weakly inhomogeneous solutions] \label{theo:weak2}
Let us consider $\lambda$ in $(0,\lambda_*]$. Consider a spatially homogeneous distribution $g_{in}=g_{in}(v) \in  L^1_v\left(m^{k_0}\right) \cap H^{k_1}$ for $k_0$, $k_1$ large enough so that it gives rise to an homogeneous solution $g\in L^\infty_t(L^1_v( m)) \cap L^1_t(L^1_v(\langle v \rangle m)) $ that satisfies $\|g_t - G_\lambda\| _{L^1_v(m)} \rightarrow 0$ and with the same global mass and momentum as $G_\lambda$. 

There is some constructive constant $\eps(g_{in})>0$ such that for any initial datum $f_{in} \in \EE$ satisfying
$$
\|f_{in}-g_{in}\|_{\EE} \leq \eps(g_{in}),
$$
and $f_{in}$ has the same mass and momentum as $G_\lambda$ and $g_{in}$, there exists a unique global solution $f \in L^\infty_t(\EE) \cap L^1_t(\EE^1)$ to~(\ref{eq:Bol1}).

Moreover, this solution satisfies
$$
\forall \, t \geq 0, \quad \| f_t - g_t\|_{\EE} \leq C \, \eps(g_{in})
$$
and for any ${\alpha} \in (0,-\mu_\lambda)$,
$$
\forall \, t \geq 0, \quad \|f_t - G_\lambda\|_{\EE} \leq C e^{-{\alpha} t}
$$
for some constructive constant $C\ge 1$.
\end{theo}

 \begin{rem}
Let us emphasize here that we are not able to get such a result in the case of a non-constant restitution coefficient because of the lack of result concerning the long-time behavior of solutions to the homogeneous corresponding problem (no result is available for now on this problem for general initial data, meaning far from the equilibrium). Consequently, in the non-constant case, we can prove such a result in a weakly inhomogeneous setting considering an homogeneous distribution $g_{in}=g_{in}(v)$ which is close enough to the equilibrium. Indeed, using Theorem \ref{theo:existence}, we obtain the existence of a solution of the equation (\ref{eq:Bol1}) which converges to the equilibrium. However, we can not conclude if we do not suppose that $g_{in}$ is close enough to $G_\lambda$.
 \end{rem}

\medskip
\subsection{Proof of the main results} \label{subsec:mainresults}
\subsubsection{Proof of Theorem~\ref{theo:existence}}
We begin by giving the key a priori estimate. 
\begin{lem} \label{lem:apriori}
With the notations of Theorem~\ref{theo:existence}, in the space $\EE$, a solution $f_t$ to the Boltzmann equation formally writes $f_t = G_\lambda + h_t$, $\Pi_{\LL_\lambda,0} h_t =0$, and $h_t$ satisfies the estimate
$$
\frac{d}{dt} \Nt h_t \Nt_{\EE} \leq (C \Nt h_t \Nt_{\EE} - K) \Nt h_t \Nt_{\EE^1}
$$
for some constants $C$, $K >0$. 
 \end{lem}

 \begin{proof}
 We consider the case $L^1_x L^1_v (m)$, we will skip the proof of other cases which is similar. 
 We have
 $$
 \frac{d}{dt} \Nt h_t \Nt_{L^1_x L^1_v ( m)} = I_1 + I_2
 $$
 with
 \begin{align*}
 I_1 := &\, \eta \int_{\R^3} \left( \int_{\T^3} \LL_\lambda h_t \, \text{sign}(h_t) \, dx \right) m \, dv \\
 & + \int_0^\infty \int_{\R^3} \left( \int_{\T^3} S_{\LL_\lambda}(\tau)(\LL_\lambda h_t) \, \text{sign}( S_{\LL_\lambda}(\tau)h_t) \, dx \right) m \, dv \, d\tau
 \end{align*}
 and
 \begin{align*}
I_2:= &\, \eta \int_{\R^3} \left( \int_{\T^3} Q_{e_\lambda}(h_t, h_t) \, \text{sign}(h_t) \, dx \right) m \, dv \\
 & + \int_0^\infty \int_{\R^3} \left( \int_{\T^3} S_{\LL_\lambda}(\tau) Q_{e_\lambda}(h_t, h_t) \, \text{sign}( S_{\LL_\lambda}(\tau)h_t) \, dx \right)m \, dv \, d\tau
\end{align*}
We already know from Proposition~\ref{prop:dissipnorm} that by choosing $\eta$ small enough, we have
$$
I_1 \leq -K \Nt h_t \Nt_{L^1_x L^1_v (\langle v \rangle m)}, \quad K>0.
$$
For the second term, we have
\begin{align*}
I_2 \leq \, & \eta \int_{\R^3} \|Q_{e_\lambda}(h_t, h_t)\|_{L^1_x( m)} \, dv + \int_0^\infty \int_{\R^3} \| S_{\LL_\lambda}(\tau)Q_{e_\lambda}(h_t, h_t) \|_{L^1_x ( m)} \, dv \, d\tau \\
\leq \, &\eta \, \|Q_{e_\lambda} (h_t, h_t)\|_{L^1_x L^1_v ( m)} + \int_0^\infty \|S_{\LL_\lambda}(\tau)Q_{e_\lambda}(h_t, h_t) \|_{L^1_x L^1_v ( m)} \, d\tau.
\end{align*}
We thus deduce 
$$
\frac{d}{dt} \Nt h_t\Nt_{L^1_x L^1_v ( m)} \leq  
-K \Nt h_t\Nt_{L^1_x L^1_v (\langle v \rangle m)} +  \Nt Q_{e_\lambda}(h_t, h_t)  \Nt_{L^1_x L^1_v ( m)}.
$$
Now using the bilinear estimate coming from Lemma~\ref{lem:bilinear}, the semigroup decay~(\ref{eq:SGlambdaalpha}) and the fact that $\Pi_{\LL_\lambda,0} Q_{e_\lambda}(h_t, h_t) = 0$, we obtain
\begin{align*}
\Nt Q_{e_\lambda}(h_t, h_t) \Nt_{L^1_x L^1_v ( m)} 
\leq &\, \eta \, \|Q_{e_\lambda}(h_t, h_t) \|_{L^1_x L^1_v ( m)} \\ 
& + \int_0^{\infty} \|S_{\LL_\lambda} (\tau) Q_{e_\lambda}(h_t, h_t) \|_{L^1_x L^1_v ( m)} \, d\tau \\
\leq &\, \eta \, \|h_t\|_{L^1_x L^1_v ( m)} \|h_t\|_{L^1_x L^1_v (\langle v \rangle m)} \\
&+ C \left( \int_0^\infty e^{-\alpha_\lambda \tau} \, d\tau \right) \|h_t\|_{L^1_x L^1_v ( m)} \|h_t\|_{L^1_x L^1_v (\langle v \rangle m)} \\
\leq &\, C\,  \|h_t\|_{L^1_x L^1_v (m)} \|h_t\|_{L^1_x L^1_v (\langle v \rangle m)} \\
\leq &\, C  \,  \Nt h_t \Nt_{L^1_x L^1_v (m)} \Nt h_t \Nt_{L^1_x L^1_v (\langle v \rangle m)},
\end{align*}
which concludes the proof. 
 \end{proof}
 
 We shall now construct solutions by considering the following iterative scheme
 $$
 \partial_t h^{n+1} = \LL_\lambda h^{n+1} + Q_{e_\lambda} (h^n, h^n), \quad n \geq1,
 $$
 with the initialization 
 $$
 \partial_t h^0 = \LL_\lambda h^0, \quad h_{in}^0 = h_{in}
 $$
 and we assume $\Nt h_{in}\Nt_{\EE} \leq \eps/2$. The functions $h^n$, $n \geq 0$ are well-defined in $\EE$ thanks to Theorem~\ref{theo:linearmainresult}. 
 
The proof is split into three steps. \\
\smallskip \noindent 
{\it Step~1. Stability of the scheme.} Let us prove by induction the following control
\beqn \label{eq:control}
\forall \, n \geq 0, \quad \sup_{t \geq 0} \left( \Nt h_t^n \Nt_{\EE} + K \int_0^t \Nt h_\tau^n\Nt_{\EE^1} \, d\tau \right) \leq \eps
\eeqn 
as soon as $\eps \leq K/(2C)$. 

The initialization is deduced from Proposition~\ref{prop:dissipnorm} and the fact that $\|h_{in}\|_{\EE} \leq \eps/2$: 
$$
\sup_{t\geq 0}  \left( \Nt h_t^0\Nt_{\EE} + K \int_0^t \Nt h_\tau^0 \Nt_{\EE^1} \, d\tau \right) \leq \eps. 
$$
Let us now assume that~(\ref{eq:control}) is satisfied for any $0 \leq n \leq N \in \N^*$ and let us prove it for $n = N+1$. A similar computation as in Lemma~\ref{lem:apriori} yields
$$
\frac{d}{dt} \Nt h ^{N+1} \Nt_{\EE} + K \| h^{N+1} \|_{\EE^1}  \leq C \Nt Q_{e_\lambda}(h^N, h^N) \Nt_{\EE}
$$
for some constants $C, K >0$, which implies
\begin{align*}
\| h_t^{N+1} \|_{\EE} + K \int_0^t \|h_\tau^{N+1} \|_{\EE^1} \, d\tau 
&\leq \Nt h_{in} \Nt_{\EE} + \int_0^t \Nt Q_{e_\lambda} ( h_\tau^N, h_\tau^N) \Nt_{\EE} \, d\tau \\
&\leq  \Nt h_{in} \Nt_{\EE} + C \left( \sup_{\tau\geq0} \Nt h_\tau ^N \Nt_{\EE} \right) \int_0^t \Nt h_\tau^N \Nt_{\EE^1} \, d\tau \\
&\leq \frac{\eps}{2} + \frac{C}{K} \eps^2 \\
&\leq \eps,
\end{align*}
as soon as $\eps < K/(2C)$. 

\smallskip \noindent
{\it Step~2. Convergence of the scheme.} Let us now denote $d^n := h^{n+1} - h^n$ and $s^n := h^{n+1} + h^n$ for $ n\geq 0$. They satisfy
$$
\forall \, n \geq 0, \quad \partial_t d^{n+1} = \LL_\lambda d^{n+1} + Q_{e_\lambda}(d^n, s^n) + Q_{e_\lambda}( s^n, d^n)
$$
and
$$
\partial_t d^0 = \LL_\lambda d^0 + Q_{e_\lambda} ( h^0, h^0).
$$
Let us denote 
$$
A^n(t) := \sup_{ 0 \leq r \leq t} \left( \Nt d_r^n\Nt_{\EE} + K \int_0^r\|d_\tau^n\|_{\EE^1} \, d\tau \right).
$$
We can prove by induction that 
$$
\forall \, t \geq 0, \quad \forall \, n \geq 0, \quad A^n(t) \leq (\overline{C} \eps)^{n+2}
$$
for some constant $\overline{C} >0$. 

Hence for $\eps$ small enough, the series $\sum_{n\geq0} A^n(t)$ is summable for any $t \geq 0$ and the sequence $h^n$ has the Cauchy property in $L^\infty_t(\EE)$, which proves the convergence of the iterative scheme. The limit $h$ as $n$ goes to infinity satisfies the equation in the strong sense in $\EE$. 

\smallskip \noindent
{\it Step~3. Rate of decay.} We now consider the solution $h$ constructed so far. From the first step, we first deduce by letting $n$ go to infinity in the stability estimate that
$$
\sup_{t\geq0} \left( \Nt h_t \Nt_{\EE} + K \int_0^t \Nt h_\tau \Nt_{\EE^1} \, d\tau \right) \leq \eps. 
$$
Second, we can apply the a priori estimate from Lemma~\ref{lem:apriori} to this solution $h$ which implies that
$$
\Nt h_t \Nt_{\EE} \leq e^{-\frac{K}{2}t} \| h_{in} \|_{\EE}
$$
under the appropriate smallness condition on $\eps$. Using the fact that $\Nt  h_t \Nt_{\EE}$ converges to zero as $t \rightarrow +\infty$, we obtain
$$
\int_t^\infty \|h_t\|_{\EE^1} \, d\tau \leq \frac{2}{K\eta} \|h_t\|_{\EE} \leq C e^{-\frac{K}{2} t } \| h_{in} \|_{\EE}.
$$
We shall now perform a bootstrap argument in order to ensure that the solution $h_t$ enjoys the same decay rate $O(e^{-\alpha t})$, $\alpha \in (0,-\mu_\lambda)$ as the linearized semigroup (Corollary~\ref{cor:linearmainresult}). Assuming that the solution is known to decay as
$$
\|h_t\|_{\EE} \leq C e^{-\alpha_0 t}
$$
for some constant $C>0$, we can prove that it indeed decays 
$$
\|h_t\|_{\EE} \leq C' e^{-\alpha_1 t}
$$
with $\alpha_1 = \text{min}\left( \alpha_0 + K/4, \alpha\right)$. It can be proved using Theorem~\ref{theo:linearmainresult} and Lemma~\ref{lem:bilinear}. Hence, in a finite number of steps, it proves the desired decay rate $O(e^{-\alpha t})$. 

\smallskip
\subsubsection{Proof of Theorem~\ref{theo:weak2}}
We split the proof into three steps. We will only deal with the case $L^1_x L^1_v (m)$.

\smallskip \noindent
{\it Step~1. The spatially homogeneous evolution.} We consider the spatially homogeneous initial datum $g_{in}$. From~\cite[Corollary~6.3]{MM2}, we know that it gives rise to a spatially homogeneous solution $g_t \in L^1_v( m)$ which satisfies 
$$
\|g_t - G_\lambda\| _{L^1_v(m)} \rightarrow 0
$$ 
with explicit exponential rate and $g\in L^\infty_t(L^1_v( m)) \cap L^1_t(L^1_v(\langle v \rangle m)) $. Let us notice that this kind of result is not available for now in the case of a non-constant resitution coefficient.

\smallskip \noindent
{\it Step~2. Local in time stability estimate.} The goal is to construct a solution $f_t$ close to some spatially homogeneous solution $g_t$ which is uniformly bounded in $L^1_xL^1_v(\langle v \rangle m)$. We consider the difference $d_t := f_t -g_t$  and we write its evolution equation:
\begin{align*}
\partial_t d + v \cdot \nabla_x d &= Q_{e_\lambda}(d,d) + Q_{e_\lambda}^+(g,d) + Q_{e_\lambda}^+(d,g) - Q_{e_\lambda}^-(g,d) - Q_{e_\lambda}^-(d,g) + \lambda^\gamma \Delta_v d \\
&= \PP(d) + \lambda^\gamma \Delta_v d,
\end{align*}
where $\PP(d) := Q_{e_\lambda}(d,d) + Q_{e_\lambda}^+(g,d) + Q_{e_\lambda}^+(d,g) - Q_{e_\lambda}^-(g,d) - Q_{e_\lambda}^-(d,g)$.
We then estimate the time evolution of the $L^1_x L^1_v( m)$ norm:
\begin{align*}
&\frac{d}{dt} \|d_t\|_{L^1_x L^1_v(  m)} = \int_{\R^3 \times \T^3} (\PP(d_t) + \lambda^\gamma \Delta_v d_t) \, \text{sign} \, d_t \, dx \, m \, dv \\
\leq &\, C \, \| Q_{e_\lambda}(d_t, d_t) \|_{L^1_x L^1_v( m)} +  C\, \| Q_{e_\lambda}^+(g_t, d_t) \|_{L^1_x L^1_v( m)}  + C\, \| Q_{e_\lambda}^+(d_t, g_t) \|_{L^1_x L^1_v(m)}   \\
&+ C\, \| Q_{e_\lambda}^-(d_t, g_t) \|_{L^1_x L^1_v(m)}  - \int_{\R^3\times\T^3} Q_{e_\lambda}^-(g_t, d_t) \, \text{sign} \, d_t \, dx \, m \, dv \\
&+ \lambda^\gamma \int_{\R^3\times\T^3} \Delta_v |d_t| \, dx \,  m \, dv.
\end{align*}

First, using the bilinear estimates of Lemma~\ref{lem:bilinear}, we have
$$
\|Q_{e_\lambda}(d,d)\|_{L^1_x L^1_v( m)} \leq C \| d\|_{L^1_x L^1_v(m)} \| d\|_{L^1_x L^1_v(\langle v \rangle m)}
$$
and 
\begin{align*}
\|Q_{e_\lambda}^+(d,g)\|_{L^1_x L^1_v( m)} + \| Q_{e_\lambda}^+(g,d)\|_{L^1_x L^1_v(m)} \leq &\,\eta \, \|g\|_{L^1_x L^1_v(\langle v \rangle m)} \|d\|_{L^1_x L^1_v(\langle v \rangle m)} \\
&+ C_\eta \, \|g\|_{L^1_x L^1_v( m)} \|d\|_{L^1_x L^1_v( m)}
\end{align*}
for any $\eta >0$ as small as wanted, and some corresponding $\eta$-dependent constant $C_\eta$. Second, by trivial explicit computations we have
$$
\|Q_{e_\lambda}^-(d,g)\|_{L^1_x L^1_v( m)} \leq C \, \|d\|_{L^1_x L^1_v( m)} \|g\|_{L^1_x L^1_v(\langle v \rangle m)}.
$$
Third, we have for some $K>0$,
$$
- \int_{\R^3\times \T^3} Q_{e_\lambda}^-(g, d) \, \text{sign} \, d_t \, dx \,  m \, dv \leq - K \|d\|_{L^1_x L^1_v(\langle v \rangle m)}.
$$
Fourth and last, 
$$
\lambda^\gamma \int_{\R^3\times \T^3} \Delta_v |d| \, dx \, m \, dv \leq C \, \| d \|_{L^1_x L^1_v( m)} \leq C \, \|d\|_{L^1_x L^1_v(\langle v \rangle m)}.
$$

Gathering all these estimates, we finally obtain 
\begin{align*} 
\frac{d}{dt} \|d_t\|_{L^1_x L^1_v( m)} \leq &\, (C \, \|d_t\|_{L^1_x L^1_v( m)} + \lambda^\gamma - K) \|d_t\|_{L^1_x L^1_v(\langle v \rangle m)} \\
&+ C \, \| g_t\|_{L^1_x L^1_v(\langle v \rangle m)} \| d_t \|_{L^1_x L^1_v( m)}. \nonumber
\end{align*} 

We then introduce an iterative scheme
$$
\partial_t d^{n+1} = Q_{e_\lambda}(d^n,d^n) + Q_{e_\lambda}(g,d^n) + Q_{e_\lambda}(d^n,g), \quad n \geq 0,
$$
and 
$$
\partial_t d^0 = Q_{e_\lambda}(g, d^0) + Q_{e_\lambda}(d^0,g)
$$
with $d_{in}^n = d_{in} = f_{in} - g_{in}$ for all $n \geq0$, just as in the previous subsection. At each step, a global solution $d_n$ is constructed in $L^1_x L^1_v( m)$ using the estimates above. We assume that $\|d_{in}\|_{L^1_x L^1_v( m)} \leq \eps/2$. By passing to the limit in the a priori estimates, we deduce that, as long as
\beqn \label{eq:smallcondition}
C \, \| d_t \|_{L^1_x L^1_v(m)} \leq K - \lambda^\gamma
\eeqn
we have
$$
\|d_t\|_{L^1_x L^1_v( m)} \leq \frac{\eps}{2} \, \text{exp} \left( C \int_0^t \|g_\tau \|_{L^1_x L^1_v(\langle v \rangle m)} \, d\tau \right) .
$$ 
We then choose $\eps$ small enough so that $ C \eps \leq K - \lambda^\gamma$, and then since \linebreak$g_t \in L^1_t\left(L^1_x L^1_v(\langle v \rangle m)\right)$, we can choose $T_1 = T_1(\eps) >0$ so that the smallness condition~(\ref{eq:smallcondition}) is satisfied and
$$
\forall \, t \in [0, T_1], \quad \|d_t\|_{L^1_x L^1_v( m)} \leq \eps.
$$ 
Observe that $T_1(\eps) \xrightarrow[\eps \rightarrow 0]{} + \infty$. This completes the proof of stability. 

\smallskip \noindent
{\it Step~3. The trapping mechanism.} Consider $\delta$ the smallness constant of the stability neighborhood in Theorem~\ref{theo:existence} in $L^1_x L^1_v( m)$. Then from~\cite{MM2}, we deduce that there is some time $T_2=T_2(M)>0$ such that
$$
\forall \, t \geq T_2, \quad \|g_t - G_{\lambda,g}\|_{L^1_x L^1_v( m)} \leq \frac{\delta}{3}
$$
where $G_{\lambda,g}$ is the equilibrium associated to $g_{in}$. We then choose $\eps$ small enough such that
$$
\| f_{in} - g_{in} \|_{L^1_x L^1_v( m)} \leq \eps \Rightarrow \|G_{\lambda,f} - G_{\lambda, g} \|_{L^1_x L^1_v( m)} \leq \frac{\delta}{3}
$$
where $G_{\lambda,f}$ is the equilibrium associated to $f_{in}$, $T_1(\eps) \geq T_2(M)$ and 
$$
\|f_{T_2} - g_{T_2}\|_{L^1_x L^1_v( m)} \leq \frac{\delta}{3},
$$
from the stability result. 

We deduce that 
\begin{align*}
\|f_{T_2} - G_{\lambda, f} \|_{L^1_x L^1_v( m)} \leq  &\, \|f_{T_2} - g_{T_2}\|_{L^1_x L^1_v( m)} + \|g_{T_2} - G_{\lambda,g}\|_{L^1_x L^1_v m)} \\
&+  \|G_{\lambda,f} - G_{\lambda, g} \|_{L^1_x L^1_v( m)} \\
\leq&\,  \delta
\end{align*}
and we can therefore use the perturbative Theorem~\ref{theo:existence} for $t \geq T_2$ which concludes the proof.

\bigskip 
\bibliographystyle{acm}

\end{document}